\documentclass{article}

\usepackage[margin=1in]{geometry}
\usepackage[parfill]{parskip}
\usepackage[utf8]{inputenc}
\usepackage{amsmath,amssymb,amsfonts,amsthm,mathrsfs,color}
\usepackage{graphicx}
\usepackage[skins]{tcolorbox}
\usepackage{bm}
\usepackage{bbm}
\usepackage{multirow,makecell}
\usepackage{subcaption}
\usepackage[colorlinks=true,linkcolor=blue,citecolor=blue]{hyperref}

\theoremstyle{definition}
\newtheorem{Theorem}{Theorem}[section]
 
\newtheorem{Lemma}[Theorem]{Lemma}

\newtheorem{Remark}[Theorem]{Remark}  
\newtheorem{Assumption}[Theorem]{Assumption} 
\newtheorem{Example}[Theorem]{Example}

\DeclareMathOperator*{\argmin}{\mathrm{argmin}}

\let\OLDthebibliography\thebibliography
\renewcommand\thebibliography[1]{
  \OLDthebibliography{#1}
  \setlength{\parskip}{1pt}
  \setlength{\itemsep}{0pt plus 0.0ex}}

\title{New closed-form estimators for discrete distributions}
\author{Adrian Fischer\footnote{Adrian Fischer, University of Oxford, UK. E-mail: adrian.fischer@stats.ox.ac.uk}}

\begin{document}
\maketitle

\begin{abstract}
    We revisit the problem of parameter estimation for discrete probability distributions with values in $\mathbb{Z}^d$. To this end, we adapt a technique called Stein's Method of Moments to discrete distributions which often gives closed-form estimators when standard methods such as maximum likelihood estimation (MLE) require numerical optimization. These new estimators exhibit good performance in small-sample settings which is demonstrated by means of a comparison to the MLE through simulation studies. We pay special attention to truncated distributions and show that the asymptotic behavior of our estimators is not affected by an unknown (rectangular) truncation domain.
\end{abstract}

\noindent{{\bf{Keywords:}}} discrete distributions; point estimation; Stein's method; truncated distributions

\section{Introduction}
Normalizing constants for discrete probability distributions often involve complicated special functions or are not available in closed-form at all. Parametric inference for these models encompasses several difficulties. If the normalizing constant is not available in closed-form, numerical approximations are necessary and can be computationally challenging. In addition, standard methods such as maximum likelihood estimation (MLE) require iterative methods and might be unfeasible, especially in a high-dimensional setting. To bypass these computational issues, \cite{betsch2022characterizations} 
proposed estimators based on a representation of the probability mass function in terms of an expectation which is independent of the normalizing constant. The authors then replace the expectation by the sample mean and measure the distance between the latter empirical average and the sample probability mass function. In \cite{lyu2009interpretation}, the score matching approach \cite{hyvarinen2007some,hyvarinen2005estimation} was generalized to discrete data by substituting the gradient in Euclidean space with linear operators from a more general class. However, their estimates still involve the normalizing constant in the univariate case and contain infinite sums if the domain of the distribution is not of finite cardinality. More recently, \cite{xu2025generalized} refined the discrete score matching approach by using a finite difference operator. \par

In this note, we extend the concept of Stein's Method of Moments as introduced in \cite{ebner2025stein} to univariate and multivariate distributions whose domain is a (rectangular) subset of $\mathbb{Z}^d$. In Stein's method one uses differential operators to obtain distributional approximations by solving a corresponding differential equation. It was originally introduced by Charles Stein \cite{stein1972bound} in 1972 for the normal distribution to derive a simple proof of the central limit theorem. Since then it has been extended to many more probability distribution, including distributions with a discrete support, we refer to the introduction of \cite{betsch2022characterizations} for some references. We note that Stein's method has been applied many times in statistics and machine learning, we refer to the survey \cite{anastasiou2023stein} and the references therein. The foundation of Stein's Method of Moments are so-called \textit{Stein operators} which allow to characterize probability distribution in the following sense: Let $X$ be a random variable and $\mathbb{P}_{\theta}$ be a probability distribution, where $\theta \in \Theta \subset \mathbb{R}^q$ is the $q$-dimensional parameter of interest. Then $\mathcal{A}_{\theta}$ is a Stein operator for $\mathbb{P}_{\theta}$ if
\begin{align*}
    \mathbb{E}[\mathcal{A}_{\theta} f(X)]=0
\end{align*}
for all functions $f$ which belong to a certain class of functions $\mathscr{F}$ if and only if $X \sim \mathbb{P}_{\theta}$. The equation from the last display is often referred to as the \textit{Stein identity}. To estimate $\theta$ from a sample $X_1,\ldots,X_n \sim \mathbb{P}_{\theta}$ we choose $q$ test functions $f_1,\ldots,f_q \in \mathscr{F}$ and replace the expectation by the sample mean which yields the system of equations
\begin{align}  \label{emp_stein_identity}
\begin{gathered} 
    \frac{1}{n} \sum_{i=1}^n \mathcal{A}_{\theta}f_1(X_i) = 0, \\
    \vdots \\
    \frac{1}{n} \sum_{i=1}^n \mathcal{A}_{\theta}f_q(X_i) = 0.
\end{gathered}
\end{align}
Solving the latter system of equations for $\theta$ then gives an estimate $\hat{\theta}_n$ to which we will refer in the following as a \textit{Stein estimator}. This approach has been proven useful for univariate distributions \cite{ebner2025stein}, multivariate truncated distributions \cite{fischer2025stein} and distributions on manifolds \cite{fischer2025steinsphere}. More recently, the approach has been applied to obtain simple estimators for the local-dependency exponential random graph model \cite{fischer2025steinexp}. A key advantage is that the Stein operator is independent of the normalizing constant and therefore often linear in the parameters of interest which leads to closed-form estimators. Therefore, Stein estimators can be used as a starting value of an iterative algorithm which computes the MLE or as the final estimate if numerical optimization is not feasible. The estimators are also suitable if the support of the distribution is truncated as the Stein estimators only differ with respect to the test functions (see Examples \ref{example_truncated_poisson}, \ref{example_truncated_binomial} and \ref{example_truncated_multneg}). An established method in order to obtain Stein operators is the \textit{density approach} (see for example \cite{ley2013stein,stein1986approximate,stein2004use}) which was also used in \cite{ebner2025stein,fischer2025stein} to derive Stein estimators. These operators are defined by
\begin{align*}
    \mathcal{A}f(x) = \frac{\nabla (f(x)p_{\theta}(x))}{p_{\theta}(x)},
\end{align*}
where $p_{\theta}$ is the density function with respect to $\mathbb{P}_{\theta}$. The density approach Stein operator can be adapted to discrete distributions on $\mathbb{Z}^d$ by a using finite difference operator $\Delta^+ f(k)=f(k+1)-f(k)$ instead of the standard gradient in Euclidean space, see for example \cite{ley2013local,betsch2022characterizations}. In the present work we use these discrete density approach Stein operators to obtain Stein estimators for discrete probability distributions. This approach has already been considered in a narrower setting in \cite{nik2024generalized}, where the authors consider a slightly different Stein identity and apply it to the negative binomial distribution. \par

In the following paragraph we briefly discuss in greater depth the two most comparable approaches which are the score matching \cite{xu2025generalized} and the minimum distance method \cite{betsch2022characterizations}. We further indicate the limitations of the present paper. The score matching estimator is defined by
\begin{align*}
    \argmin_{\theta} \bigg\{ \frac{1}{n} \sum_{i=1}^n \iota \bigg( \frac{p_{\theta}(X_i+1)}{p_{\theta}(X_i)} \bigg)^2 + \iota\bigg( \frac{p_{\theta}(X_i)}{p_{\theta}(X_i-1)} \bigg)^2 - 2\iota \bigg( \frac{ p_{\theta}(X_i+1)}{p_{\theta}(X_i)}  \bigg) \bigg\},
\end{align*}
where $\iota(u)=1/(1+u)$, and the minimum distance estimator is given by
\begin{align*}
     \argmin_{\theta} \bigg\{\sum_{k} \bigg( \frac{1}{n} \sum_{i=1}^n \mathbbm{1}\{X_i=k\} + \frac{p_{\theta}(X_i+1)-p_{\theta}(X_i)}{p_{\theta}(X_i)} \mathbbm{1}\{X_i>k\} \bigg)^2 \bigg\},
\end{align*}
where the sum with respect to $k$ ranges over the domain of the distribution. As can be readily seen, both estimators are suitable for unnormalized models. However, both estimators are often not available in closed-form (see Example \ref{example_yules_simon}), even for simple univariate probability distributions. This necessitates the use of numerical procedures to compute the estimate which can be unfeasible in a high-dimensional setting. Moreover, the minimum distance estimator involves a summation ranging from the lowest to the largest observed value which can be computationally challenging and the authors do not provide a generalization to multivariate distributions. The fact that both estimation techniques above, although independent of the normalizing constant, do often require numerical computation was the main motivation for this paper: For all examples in this manuscript the Stein estimator can be computed in closed-form. However, also our approach bears several limitations: For some discrete probability distributions, also the Stein's methods of Moments approach does not yield closed-form estimators. For instance, for the multivariate examples in Section \ref{section_multivariate} we limit ourselves to estimation of a subset of parameters of the model and assume the remaining parameters to be known. Furthermore, explicitness of the estimator is accompanied by the forfeiture of asymptotic efficiency. Nonetheless, the performance in the simulation studies suggests that the loss of efficiency is only minor. This is also supported by Example \ref{example_logarithmic} where we compute the asymptotic variance of the Stein estimator for the logarithmic distribution and compare it to the asymptotically efficient MLE. Moreover, in Example \ref{example_yules_simon} (where we consider the Yules-Simon distribution), we included the score matching and minimum distance approach in the simulation study and observe that the Stein estimator achieves higher performance on an average level. \par

The rest of the paper is organized as follows: In Section \ref{section_estimation} we introduce the discrete density approach Stein operator together with a suitable function class which allow for statistical inference. We further explore several univariate examples and compare our new estimator to the maximum likelihood method through simulation studies in terms of bias and mean squared error (MSE). In Section \ref{section_multivariate} we extend the concept to multivariate distributions and work out two more examples. In Section \ref{section_unknwon_truncation} we investigate truncated distributions on a rigorous level: We show that the asymptotic covariance of the Stein estimator is not affected if the truncated domain has to be estimated (under the assumption that the support is a rectangle in $\mathbb{Z}^d$).

\section{Estimation} \label{section_estimation}

We introduce the discrete density approach Stein operator by using a finite difference operator instead of the gradient (see \cite{ley2013local,betsch2022characterizations}). Let us consider probability distribution with support $U = \{a, \ldots, b\} \subset \mathbb{Z}$, where $ -\infty \leq a < b \leq \infty$ and probability mass function $p_{\theta}(k)$, $ k \in U$ with corresponding probability measure $\mathbb{P}_{\theta}$ where $\theta \in \Theta \subset \mathbb{R}^q$. We define the operator $\Delta^{+}f(k) = f(k+1) - f(k)$ for a function $f: \{a, \ldots, b+1 \} \rightarrow \mathbb{R}$, where we set $f(x)=0$ if $x \notin U$. In addition, let $\tau_{\theta}$ be a function defined on $U$. Then, the Stein operator is defined as
\begin{align} \label{def_stein_op}
    \mathcal{A}_{\theta} f(k) = \frac{\Delta^{+} (f(k) \tau_{\theta}(k) p_{\theta}(k)) }{p_{\theta}(k)} = \frac{f(k+1) p_{\theta}(k+1) \tau_{\theta}(k+1) }{p_{\theta}(k)} - f(k) \tau_{\theta}(k), \qquad k \in  \{a, \ldots, b \}
\end{align}
with the convention that $p_{\theta}(b+1)$ and $ \tau_{\theta}(b+1)$ are equal to some arbitrary strictly positive number if $b < \infty$. We now introduce the function class $\mathscr{F}_{\theta}$ which is defined as 
\begin{align} \label{def_stein_class}
\begin{split}
    \mathscr{F}_{\theta}= \big\{ f: \{a, \ldots, b+1\} \rightarrow \mathbb{R} \, \vert  \, & f(a)=0 \text{ if } a>-\infty, \lim_{k \rightarrow - \infty} f(k) \tau_{\theta}(k)p_{\theta}(k) = 0 \text{ if } a= -\infty \\
    &\text{ and }  f(b+1)=0 \text{ if } b<\infty, \lim_{k \rightarrow  \infty} f(k) \tau_{\theta}(k)p_{\theta}(k) = 0 \text{ if } b= \infty \big\}.
\end{split}
\end{align}
Now define $\mathscr{F} = \cap_{\theta \in \Theta} \mathscr{F}_{\theta}$. The next theorem states that the expectation under $\mathbb{P}_{\theta}$ of the Stein operator \eqref{def_stein_op} applied to functions from the class \eqref{def_stein_class} indeed cancels out. Similar results were already proven in \cite[Theorem 2.1]{ley2013local} and \cite[Theorem 2.2]{betsch2022characterizations} under slightly different assumptions on the function class $\mathscr{F}_{\theta}$. For the proof, we refer to Theorem \ref{theorem_char_stein_mult} where the result is shown for multivariate distributions.
\begin{Theorem} \label{theorem_char_stein_uni}
    Let $X \sim \mathbb{P}_{\theta}$. Then,
    \begin{align*}
        \mathbb{E}[\mathcal{A}_{\theta} f(X) ] = 0
    \end{align*}
    for all $f \in \mathscr{F}$.
\end{Theorem}
\begin{Remark}
    One can show that the function class $\mathscr{F}_{\theta}$ is characterizing for $\mathbb{P}_{\theta}$. To see this, one follows along the lines of the proofs in \cite[Theorem 2.1]{ley2013local} and \cite[Theorem 2.2]{betsch2022characterizations}: Let $X$ be a discrete random variable with support $U$ and probability mass function $p_X$ such that $\mathbb{E}[\mathcal{A}_{\theta} f_{\theta,m}(X) ]=0$. Now let $Z \sim \mathbb{P}_{\theta}$ and define for $m \in U$,
    \begin{align*}
        f_{\theta,m}(k)=\frac{1}{p_{\theta}(k)\tau_{\theta}(k)} \sum_{l=a}^{k-1} \big( \mathbbm{1}\{l \leq m \}- \mathbb{P}_{\theta}(Z\leq m) \big)p_{\theta}(l), \qquad k \in U
    \end{align*}
    and $f_{\theta,m}(b+1)=0$ if $b<\infty$. If $a>-\infty$ then we have  $f_{\theta,m}(a)=0$ as we sum over an empty set. If $a=-\infty$ we have
    \begin{align*}
        \lim_{k \rightarrow - \infty} f_{\theta,m}(k) \tau_{\theta}(k)p_{\theta}(k) = \lim_{k \rightarrow - \infty} \sum_{l=- \infty}^{k-1} \big( \mathbbm{1}\{l \leq m \}- \mathbb{P}_{\theta}(Z\leq m) \big)p_{\theta}(l) =0
    \end{align*}
    since the sum $\sum_{l=- \infty}^{b} \big( \mathbbm{1}\{l \leq m \}- \mathbb{P}_{\theta}(Z\leq m) \big)p_{\theta}(l)$ is finite. If $b<\infty$ we have $f_{\theta,m}(b+1)=0$ by definition and if $b=\infty$ we compute
    \begin{align*}
        &\lim_{k \rightarrow  \infty} f_{\theta,m}(k) \tau_{\theta}(k)p_{\theta}(k) = \lim_{k \rightarrow \infty} \sum_{l=a}^{k-1} \big( \mathbbm{1}\{l \leq m \}- \mathbb{P}_{\theta}(Z\leq m) \big)p_{\theta}(l) \\
         & =\sum_{l=- \infty}^{m} p_{\theta}(l) -   \mathbb{P}_{\theta}(Z\leq m) \lim_{k \rightarrow - \infty} \sum_{l=- \infty}^{k-1} p_{\theta}(l) = 0.
    \end{align*}
    Therefore, $f_{\theta,m} \in \mathscr{F}_{\theta}$ and 
    \begin{align*}
        0 = \mathbb{E}[\mathcal{A}_{\theta} f_{\theta,m}(X) ] &= \sum_{k=a}^b \frac{p_X(k)}{p_{\theta}(k)} \bigg( \sum_{l=a}^{k} \big( \mathbbm{1}\{l \leq m \}- \mathbb{P}_{\theta}(Z\leq m) \big)p_{\theta}(l)  \\
        & \qquad \qquad \qquad \qquad \qquad\qquad\qquad - \sum_{l=a}^{k-1} \big( \mathbbm{1}\{l \leq m \}- \mathbb{P}_{\theta}(Z\leq m) \big)p_{\theta}(l)  \bigg) \\
        &= \sum_{k=a}^b \frac{p_X(k)}{p_{\theta}(k)} \big( \mathbbm{1}\{k \leq m \}- \mathbb{P}_{\theta}(Z\leq m) \big)p_{\theta}(k) \\
        &= \sum_{k=a}^b p_X(k) \big( \mathbbm{1}\{k \leq m \}- \mathbb{P}_{\theta}(Z\leq m) \big) \\
        &= \mathbb{P}(X\leq m) - \mathbb{P}_{\theta}(Z\leq m)
    \end{align*}
    and hence $X \sim \mathbb{P}_{\theta}$.
\end{Remark}

\begin{Remark}
We could define the Stein operator also using the backward discrete derivative $\Delta^{-}f(k) = f(k) - f(k-1)$ by 
\begin{align*}
    \mathcal{A}_{\theta} f(k) = \frac{\Delta^{-} (f(k) \tau_{\theta}(k) p_{\theta}(k)) }{p_{\theta}(k)}, \qquad k \in  \{a, \ldots, b \}.
\end{align*}
For finite $a$ or $b$ we would then have the conditions $f(a-1) = 0$ resp.\ $f(b) = 0$ where the test functions are now defined on $\{a-1, \ldots, b \}$.
\end{Remark}

Let us recall the definition of a Stein estimator from the introduction. To this end, let $X_1, \ldots, X_n \sim \mathbb{P}_{\theta^{\star}}$ be independently and identically distributed (i.i.d.) on a common probability space $(\Omega,\mathcal{F},\mathbb{P})$, whereby $\mathbb{P}_{\theta^{\star}}$ is a discrete probability distribution with probability mass function $p_{\theta}$ with support $U=\{a,\ldots,b\}$ where $-\infty \leq a < b \leq \infty$. Hence, in this setting $\theta^{\star}$ denotes the true unknown parameter which is to be estimated. The Stein estimator is then the solution $\hat{\theta}_n$ to \eqref{emp_stein_identity} where the Stein operator $\mathcal{A}_{\theta}$ is given by \eqref{def_stein_op} and the test functions $f_1,\ldots,f_q$ are chosen from $\mathscr{F}$. \par

In \cite{ebner2025stein} we work out conditions for existence, consistency and asymptotic normality of the Stein estimator $\hat{\theta}_n$ for continuous univariate distributions. These results translate easily to the case of univariate or multivariate discrete distributions and are satisfied for all examples we consider in Sections \ref{section_estimation} and \ref{section_multivariate}. Accordingly, we refrain from presenting a detailed derivation. However, we briefly comment on consistency and asymptotic normality in the next remark and prove that the asymptotic variance is not affected if $a$ and $b$ are unknown in Section \ref{section_unknwon_truncation}.

\begin{Remark} \label{remark_asymptotic_normality}
Suppose that the following holds (compare \cite[Assumption 2.1]{ebner2025stein}):
    \begin{description} 
        \item[(a)] Let $X \sim \mathbb{P}_{\theta^{\star}}$ and $\theta \in \Theta$. Then $f_1, \ldots, f_q \in \mathscr{F}$ are such that $\mathbb{E}[\mathcal{A}_{\theta}f(X)]=0$ if and only if $\theta=\theta^{\star}$, where $\mathcal{A}_{\theta}f(k)=(\mathcal{A}_{\theta}f_1(k),\ldots, \mathcal{A}_{\theta}f_q(k))^{\top}$.
        \item[(b)] Let $\tilde{q} \geq q$ and $X \sim \mathbb{P}_{\theta^{\star}}$. We can write $\mathcal{A}_{\theta}f(k) =M(k)g(\theta)$ for some measurable $q \times \tilde{q}$ matrix $M$ with $\mathbb{E}[\Vert M(X) \Vert] < \infty $ and a continuously differentiable function $g=(g_1,\ldots, g_{\tilde{q}})^\top:\Theta \rightarrow \mathbb{R}^{\tilde{q}} $ for all $\theta \in \Theta$, $k \in \{a,\ldots,b\}$. We also assume that $\mathbb{E}[M(X)]\frac{\partial}{\partial \theta} g(\theta) \vert_{\theta=\theta^{\star}}$ is invertible, where $\frac{\partial}{\partial \theta} g(\theta) \vert_{\theta=\theta^{\star}}$ denotes the gradient of $g$ with respect to $\theta$, evaluated at $\theta=\theta^{\star}$.
    \end{description}
    We then have that $\hat{\theta}_n$ exists with probability converging to $1$ and is consistent for $\theta^{\star}$. If moreover, $\mathbb{E}[\Vert M(X) \Vert^2] < \infty $, we have that
    \begin{align*}
        \sqrt{n}(\hat{\theta}_n-\theta^{\star}) \xrightarrow{D} N(0,\Sigma), 
    \end{align*}
    where $\xrightarrow{D}$ denotes convergence in distribution and
    \begin{align*}
        \Sigma =  \mathbb{E}\bigg[ \frac{\partial}{\partial \theta} \mathcal{A}_{\theta} f(X) \Big\vert_{\theta = \theta^{\star}} \bigg]^{-1} \mathbb{E}\bigg[ \big( \mathcal{A}_{\theta^{\star}} f(X) \big)^2 \bigg] \mathbb{E}\bigg[ \frac{\partial}{\partial \theta} \mathcal{A}_{\theta} f(X) \Big\vert_{\theta = \theta^{\star}} \bigg]^{-\top}
    \end{align*}
    (see Theorem 2.2 and Theorem 2.3 in \cite{ebner2025stein}).
\end{Remark}

We provide several examples of univariate discrete probability distributions to which our estimation technique can be applied to. In the first two examples we consider the Poisson and binomial distribution and show that the standard estimators can be obtained by a proper choice of the test function. In what follows we write $\overline{f(X)}=\frac{1}{n} \sum_{i=1}^n f(X_i)$ for a function $f: U \rightarrow \mathbb{R}$.

\begin{Example} \label{example_poisson}
    For the Poisson distribution $P(\lambda)$ on $U=\mathbb{N}_0$ with parameter $\theta=\lambda > 0$ we have $p_{\theta}(k) = \frac{e^{-\lambda}\lambda^k}{k!}$, $\tau_{\theta}(k)=k$ and we get
    \begin{align} \label{poisson_stein_operator}
        \mathcal{A}_{\theta} f(k) = \lambda f(k+1) - k f(k). 
    \end{align}
    For an i.i.d.\ sample $X_1,\ldots,X_n \sim P(\lambda^{\star})$ the Stein estimator is given by
    \begin{align} \label{poisson_stein_estimator}
        \hat{\lambda}_n = \frac{\overline{Xf(X)}}{\overline{f(X+1)}}.
    \end{align}
    The test function choice $f(k)=1$ gives the standard estimator $\hat{\lambda}_n=\overline{X}$.
\end{Example}

\begin{Example} \label{example_binomial}
    For the Binomial distribution $B(m,p)$ on $U=\{0, \ldots, m\}$ with parameter $\theta=p \in (0,1)$ and $m \in \mathbb{N}_0$ we have $p_{\theta}(k) = \binom{m}{k} p^k (1-p)^{m-k}$ and $\tau_{\theta}(k)=(1-p)/p$ and we get
    \begin{align} \label{binomial_stein_operator}
        \mathcal{A}_{\theta} f(k) = \frac{m-k}{(k+1)}f(k+1) - \frac{1-p}{p}f(k) .   
    \end{align}
    For an i.i.d.\ sample $X_1,\ldots,X_n \sim B(m,p^{\star})$ we obtain the estimator
    \begin{align} \label{binomial_stein_estimator}
        \hat{p}_n = \bigg(1+ \frac{\overline{(m-X)f(X+1)/(X+1)}}{\overline{f(X)}} \bigg)^{-1},
    \end{align}
    where we can ignore the restriction $f(m+1)=0$ as $m-X_i=0$ if $X_i=m$. The standard estimator $\hat{p}_n=\overline{X}/m$ is recovered by the test function choice $f(x)=x$.
\end{Example}

\begin{Example} \label{example_yules_simon}
    For the Yule-Simon distribution $YS(\rho)$ on $U= \mathbb{N}$ with parameter $\theta=\rho > 0$ we have $p_{\theta}(k) = \rho B(k,\rho+1) $ (here, $B(\cdot,\cdot)$ is the Beta function) and  $\tau_{\theta}(k)=k+\rho$ we get
    \begin{align*}
        \mathcal{A}_{\theta} f(k) = k f(k+1) - (k+\rho) f(k).
    \end{align*}
    For an i.i.d.\ sample $X_1,\ldots,X_n \sim YS(\rho^{\star})$ the estimator is therefore given by
    \begin{align*}
        \hat{\rho}_n = \frac{\overline{Xf(X+1)} - \overline{Xf(X)}}{\overline{f(X)}}
    \end{align*}
    The score matching estimator is here given by
    \begin{align*}
        \hat{\theta}_n^{\mathrm{SM}} = \argmin_{\rho>0} \bigg\{ \frac{1}{n} \sum_{i=1}^n \iota \bigg( \frac{X_i}{X_i+1+\rho} \bigg)^2 + \iota\bigg( \frac{X_i-1}{X_i+\rho} \bigg)^2 - 2\iota \bigg( \frac{X_i}{X_i+1+\rho}  \bigg) \bigg\}
    \end{align*}
    and for the minimum distance estimator we obtain
    \begin{align*}
        \hat{\theta}_n^{\mathrm{MD}} = \argmin_{\rho>0} \bigg\{\sum_{k=1}^{\infty} \bigg( \frac{1}{n} \sum_{i=1}^n \mathbbm{1}\{X_i=k\} + \bigg( \frac{X_i}{X_i+1+\rho} -1 \bigg) \mathbbm{1}\{X_i>k\} \bigg)^2 \bigg\}.
    \end{align*}
    We compare the Stein estimator $\hat{\theta}_n^{\mathrm{ST}}$ to the MLE $\hat{\theta}_n^{\mathrm{ML}}$, the score matching estimator from \cite{xu2025generalized} $\hat{\theta}_n^{\mathrm{SM}}$ and the minimum distance estimator from \cite{betsch2022characterizations} $\hat{\theta}_n^{\mathrm{MD}}$ by means of a simulation study in R \cite{R}. Simulation results can be found in Table \ref{yulesimon_sim}, we report bias, MSE and we also included a column \textit{NE} (not existent) which reports the average relative frequency (out of $100$) of estimates which were not eligible (for example if the estimate was negative or the code threw an error due to failure of the numerical optimization). Regarding the Stein estimator we used the test function choice $f(k) = \log(k)$. With respect to $\hat{\theta}_n^{\mathrm{MD}}$, we excluded the first three parameter values from the simulation since evaluating the target function was computationally unfeasible (due to small and large realizations in the sample). Numerical optimization for $\hat{\theta}_n^{\mathrm{ML}}$, $\hat{\theta}_n^{\mathrm{SM}}$ and $\hat{\theta}_n^{\mathrm{MD}}$ was performed using the L-BFGS-B option in the R \texttt{optim} function with starting value $\rho=1$. One observes that all four estimators perform similarly in terms of bias and $\hat{\theta}_n^{\mathrm{ML}}$, $\hat{\theta}_n^{\mathrm{ST}}$ outperform $\hat{\theta}_n^{\mathrm{SM}}$, $\hat{\theta}_n^{\mathrm{MD}}$ throughout all parameter constellation regarding the MSE. We also note that $\hat{\theta}_n^{\mathrm{ML}}$ and $\hat{\theta}_n^{\mathrm{ST}}$ show a very similar behavior (except for $\rho=0.1$).
\end{Example}

\begin{table} 
\centering
\begin{tabular}{cc|cccc|cccc|cccc}
 $\rho^{\star}$ & & \multicolumn{4}{c|}{Bias} & \multicolumn{4}{c|}{MSE} & \multicolumn{4}{c}{NE} \\ \hline
   & & $\hat{\theta}_n^{\mathrm{ML}}$ & $\hat{\theta}_n^{\mathrm{SM}}$ & $\hat{\theta}_n^{\mathrm{MD}}$ & $\hat{\theta}_n^{\mathrm{ST}}$ & $\hat{\theta}_n^{\mathrm{ML}}$ & $\hat{\theta}_n^{\mathrm{SM}}$ & $\hat{\theta}_n^{\mathrm{MD}}$ & $\hat{\theta}_n^{\mathrm{ST}}$ & $\hat{\theta}_n^{\mathrm{ML}}$ & $\hat{\theta}_n^{\mathrm{SM}}$ & $\hat{\theta}_n^{\mathrm{MD}}$ & $\hat{\theta}_n^{\mathrm{ST}}$ \\ \hline
\multirow{1}{*}{$0.1$} & $\rho$  & $6.29\text{e-4}$ & $0.363$ & -- & $0.078$ & $1.79\text{e-4}$ & $0.598$ & -- & 0.379 & \multirow{1}{*}{4} & \multirow{1}{*}{0} & \multirow{1}{*}{--} & \multirow{1}{*}{68} \\  \hline 
\multirow{1}{*}{$0.5$} & $\rho$  & $0.013$ & $0.092$ & -- & $0.013$ & $7.07\text{e-3}$ & $0.286$ & -- & 7.04\text{e-3} & \multirow{1}{*}{0} & \multirow{1}{*}{0} & \multirow{1}{*}{--} & \multirow{1}{*}{0} \\  \hline 
\multirow{1}{*}{$0.9$} & $\rho$  & $0.032$ & $0.064$ & -- & $0.03$ & $0.03$ & $0.374$ & -- & 0.03 & \multirow{1}{*}{0} & \multirow{1}{*}{0} & \multirow{1}{*}{--} & \multirow{1}{*}{0} \\  \hline 
\multirow{1}{*}{$1$} & $\rho$  & $0.038$ & $0.069$ & $0.06$ & $0.036$ & $0.04$ & $0.404$ & $0.133$ & 0.039 & \multirow{1}{*}{0} & \multirow{1}{*}{0} & \multirow{1}{*}{0} & \multirow{1}{*}{0} \\  \hline 
\multirow{1}{*}{$1.5$} & $\rho$  & $0.075$ & $0.094$ & $0.092$ & $0.07$ & $0.124$ & $0.567$ & $0.268$ & 0.124 & \multirow{1}{*}{0} & \multirow{1}{*}{0} & \multirow{1}{*}{0} & \multirow{1}{*}{0} \\  \hline 
\multirow{1}{*}{$2$} & $\rho$  & $0.134$ & $0.13$ & $0.14$ & $0.122$ & $0.292$ & $0.84$ & $0.514$ & 0.292 & \multirow{1}{*}{0} & \multirow{1}{*}{0} & \multirow{1}{*}{0} & \multirow{1}{*}{0} \\  \hline 
\multirow{1}{*}{$2.5$} & $\rho$  & $0.197$ & $0.19$ & $0.206$ & $0.182$ & $0.634$ & $1.31$ & $0.985$ & 0.636 & \multirow{1}{*}{0} & \multirow{1}{*}{0} & \multirow{1}{*}{0} & \multirow{1}{*}{0} \\  \hline 
\multirow{1}{*}{$3$} & $\rho$  & $0.279$ & $0.249$ & $0.281$ & $0.257$ & $1.31$ & $2.09$ & $1.8$ & 1.31 & \multirow{1}{*}{0} & \multirow{1}{*}{0} & \multirow{1}{*}{0} & \multirow{1}{*}{0} \\  \hline 
\multirow{1}{*}{$3.5$} & $\rho$  & $0.372$ & $0.313$ & $0.36$ & $0.341$ & $1.93$ & $2.71$ & $2.56$ & 1.92 & \multirow{1}{*}{0} & \multirow{1}{*}{0} & \multirow{1}{*}{0} & \multirow{1}{*}{0} \\  \hline 
\multirow{1}{*}{$4$} & $\rho$  & $0.504$ & $0.433$ & $0.493$ & $0.469$ & $3.49$ & $4.31$ & $4.39$ & 3.46 & \multirow{1}{*}{0} & \multirow{1}{*}{0} & \multirow{1}{*}{0} & \multirow{1}{*}{0} \\  \hline 
\end{tabular} 
\caption{\protect\label{yulesimon_sim} Simulation results for the $YS(\rho)$ distribution for $n=50$ and $10{,}000$ repetitions.}
\end{table}

\begin{Example}
    For the Beta negative binomial distribution $BNB(\alpha,\beta,r)$ on $U= \mathbb{N}_0$ with parameter $\theta = (\alpha,\beta)$, where $\alpha, \beta> 0$ and $r >0$ we have $p_{\theta}(k) = \frac{B(r+k, \alpha + \beta)}{B(r, \alpha)} \frac{\Gamma(k+\beta)}{k! \Gamma(\beta)} $  (here $\Gamma(\cdot)$ is the gamma function). With $\tau_{\theta}(k)=(r+k+\alpha + \beta -1 )k$ we get
    \begin{align*}
        \mathcal{A}_{\theta} f(k) = (r+k)(k+\beta) f(k+1) - (r+k+\alpha + \beta -1 )kf(k).
    \end{align*}
    For an i.i.d.\ sample $X_1,\ldots X_n \sim BNB(\alpha^{\star},\beta^{\star},r)$ (where we assume $r>0$ to be known) and two test functions $f_1,f_2$ we get the Stein estimators
    \begin{align*}
        \hat{\alpha}_n&= \frac{M_n^{(3)} M_n^{(5)} - M_n^{(2)} M_n^{(6)}}{M_n^{(2)} M_n^{(4)} - M_n^{(1)} M_n^{(5)}}, \\
        \hat{\beta}_n&= \frac{M_n^{(1)} M_n^{(6)} - M_n^{(3)} M_n^{(4)}}{M_n^{(2)} M_n^{(4)} - M_n^{(1)} M_n^{(5)}},
    \end{align*}
    where 
    \begin{gather*}
        M_n^{(1)} = -\overline{Xf_1(X)}, \qquad  M_n^{(2)} = \overline{(r+X)f_1(X+1)}-\overline{Xf_1(X)} \\
        M_n^{(4)} = -\overline{Xf_2(X)}, \qquad  M_n^{(5)} = \overline{(r+X)f_2(X+1)}-\overline{Xf_2(X)}
    \end{gather*}
    and
    \begin{gather*}
        M_n^{(3)} = \overline{(r+X)X f_1(X+1)} - \overline{(r+X-1)Xf_1(X)}, \\
        M_n^{(6)} = \overline{(r+X)X f_2(X+1)} - \overline{(r+X-1)Xf_2(X)}.
    \end{gather*}
    Simulation results can be found in Table \ref{betanegbin_sim} for the Stein estimator $\hat{\theta}_n^{\mathrm{ST}}=(\hat{\alpha}_n^{\mathrm{ST}}, \hat{\beta}_n^{\mathrm{ST}})$ and the MLE $\hat{\theta}_n^{\mathrm{ML}}=(\hat{\alpha}_n^{\mathrm{ML}}, \hat{\beta}_n^{\mathrm{ML}})$. For the Stein estimator we used the test functions $f_1(k) = k$ and $f_2(k) = 1$. In order to exclude outliers for the estimation of bias and MSE we considered the truncated mean for both quantities with respect to the truncation threshold $10$, i.e.\ if $\vert \hat{\alpha}_n - \alpha \vert > 10$ or $\vert \hat{\beta}_n - \beta \vert > 10$, we considered the estimate as non-eligible for both estimators. Numerical optimization for the MLE was performed using the Nelder-Mead algorithm as implemented in the R function \texttt{optim}. We see that the MLE yields overall the best results, however, the Stein estimator achieves a lower bias for some parameter constellations.
\end{Example}

\begin{table} 
\centering
\begin{tabular}{cc|cc|cc|cc}
 $(\alpha^{\star},\beta^{\star},r)$ & & \multicolumn{2}{c|}{Bias} & \multicolumn{2}{c|}{MSE} & \multicolumn{2}{c}{NE} \\ \hline
& & $\hat{\theta}_n^{\mathrm{ML}}$ & $\hat{\theta}_n^{\mathrm{ST}}$ & $\hat{\theta}_n^{\mathrm{ML}}$ & $\hat{\theta}_n^{\mathrm{ST}}$ & $\hat{\theta}_n^{\mathrm{ML}}$ & $\hat{\theta}_n^{\mathrm{ST}}$ \\ \hline
\multirow{2}{*}{$(10,10,10)$} & $\alpha$  & $0.342$ & $0.553$ & $3.08$ & $4.31$ & \multirow{2}{*}{$0$} & \multirow{2}{*}{$0$} \\ & $\beta$  & $0.352$ & $0.582$ & $3.4$ & $4.82$ \\ \hline 
\multirow{2}{*}{$(15,20,5)$} & $\alpha$  & $-0.195$ & $-0.095$ & $10.8$ & $11$ & \multirow{2}{*}{$15$} & \multirow{2}{*}{$17$} \\ & $\beta$  & $-0.307$ & $-0.167$ & $20.9$ & $21.3$ \\ \hline 
\multirow{2}{*}{$(8,10,3)$} & $\alpha$  & $0.43$ & $0.718$ & $4.8$ & $5.87$ & \multirow{2}{*}{$4$} & \multirow{2}{*}{$5$} \\ & $\beta$  & $0.579$ & $0.987$ & $8.94$ & $11$ \\ \hline 
\multirow{2}{*}{$(3,5,7)$} & $\alpha$  & $0.056$ & $0.421$ & $0.116$ & $0.46$ & \multirow{2}{*}{$0$} & \multirow{2}{*}{$0$} \\ & $\beta$  & $0.101$ & $0.98$ & $0.41$ & $2.19$ \\ \hline 
\multirow{2}{*}{$(3,3,2)$} & $\alpha$  & $0.136$ & $0.621$ & $0.371$ & $1.01$ & \multirow{2}{*}{$0$} & \multirow{2}{*}{$0$} \\ & $\beta$  & $0.16$ & $0.863$ & $0.536$ & $1.73$ \\ \hline 
\multirow{2}{*}{$(7,5,6)$} & $\alpha$  & $0.277$ & $0.545$ & $1.79$ & $2.84$ & \multirow{2}{*}{$0$} & \multirow{2}{*}{$0$} \\ & $\beta$  & $0.208$ & $0.428$ & $1.05$ & $1.73$ \\ \hline 
\multirow{2}{*}{$(15,14,8)$} & $\alpha$  & $0.52$ & $0.683$ & $10.9$ & $12.4$ & \multirow{2}{*}{$4$} & \multirow{2}{*}{$5$} \\ & $\beta$  & $0.506$ & $0.67$ & $10.4$ & $11.9$ \\ \hline 
\multirow{2}{*}{$(10,10,3)$} & $\alpha$  & $0.493$ & $0.77$ & $8.92$ & $10.1$ & \multirow{2}{*}{$7$} & \multirow{2}{*}{$9$} \\ & $\beta$  & $0.523$ & $0.829$ & $10.3$ & $11.8$ \\ \hline 
\multirow{2}{*}{$(4,8,2)$} & $\alpha$  & $0.225$ & $0.627$ & $0.887$ & $1.57$ & \multirow{2}{*}{$1$} & \multirow{2}{*}{$2$} \\ & $\beta$  & $0.524$ & $1.58$ & $4.92$ & $9.32$ \\ \hline 
\multirow{2}{*}{$(9,9,4)$} & $\alpha$  & $0.554$ & $0.849$ & $6.16$ & $7.69$ & \multirow{2}{*}{$2$} & \multirow{2}{*}{$2$} \\ & $\beta$  & $0.587$ & $0.916$ & $7.08$ & $8.94$ \\ \hline 
\end{tabular} 
\caption{\protect\label{betanegbin_sim} Simulation results for the $BNB(\alpha,\beta,r)$ distribution for $n=300$ and $10{,}000$ repetitions.}
\end{table}

\begin{Example} \label{example_logarithmic}
    For the logarithmic distribution $LG(p)$ on $U= \mathbb{N}$ with parameter $\theta=p$, where $0< p < 1$  we have $p_{\theta}(k) = \frac{-1}{\log(1-p)} \frac{p^k}{k}$. With $\tau_{\theta}(k)=1$ we get
    \begin{align*}
        \mathcal{A}_{\theta} f(k) =  p \frac{k}{k+1}f(k+1) - f(k). 
    \end{align*}
    For an i.i.d.\ sample $X_1,\ldots,X_n \sim LG(p^{\star})$ this gives the Stein estimator
    \begin{align*}
        \hat{p}_n = \frac{\overline{f(X)}}{\overline{Xf(X+1)/(X+1)}}.
    \end{align*}
Simulation results can be found in Table \ref{logarthmic_sim} for the Stein estimator $\hat{\theta}_n^{\mathrm{ST}}$ and the MLE $\hat{\theta}_n^{\mathrm{ML}}$ whereby we used the test function $f(k)=k-1$ for the Stein estimator. As per calculation of the MLE, we used the explicit formula 
\begin{align*}
    \hat{p}_n^{\mathrm{ML}} = 1- \exp\bigg( W_{-1} \bigg( -\frac{1}{\overline{X} \exp(1/\overline{X}) } \bigg) + \frac{1}{\overline{X}} \bigg)
\end{align*}
provided in \cite{rodrigues2025closed}. In the latter expression $W_{-1}(\cdot)$ denotes the Lambert W-function. From Table \ref{logarthmic_sim} one observes that the performance of Stein estimator and MLE is very similar, for both bias and MSE. This observation is also confirmed by a plot of the relative asymptotic efficiency in Figure \ref{fig:logarithmic_variance}. Letting $X \sim LG(p)$ the asymptotic variances are given by
    \begin{align*}
        V^{\mathrm{ML}}(p)= \frac{(1-p)^2 p (\log(1-p))^2}{-p-\log(1-p)}
    \end{align*}
    for the MLE and
    \begin{align*}
        V^{\mathrm{ST}}(p)=\mathbb{E}\bigg[\bigg( \frac{pX^2}{X+1}-X+1  \bigg)^2 \bigg] \mathbb{E}\bigg[ \frac{X^2}{X+1}2 \bigg]^{-2}
    \end{align*}
    for the Stein estimator (see Remark \ref{remark_asymptotic_normality}). As one can observe in the plot, the relative asymptotic efficiency of the Stein estimator is close to $1$ across the whole parameter space.
\end{Example}

\begin{table} 
\centering
\begin{tabular}{cc|cc|cc|cc}
 $p^{\star}$ & & \multicolumn{2}{c|}{Bias} & \multicolumn{2}{c|}{MSE} & \multicolumn{2}{c}{NE} \\ \hline
& & $\hat{\theta}_n^{\mathrm{ML}}$ & $\hat{\theta}_n^{\mathrm{ST}}$ & $\hat{\theta}_n^{\mathrm{ML}}$ & $\hat{\theta}_n^{\mathrm{ST}}$ & $\hat{\theta}_n^{\mathrm{ML}}$ & $\hat{\theta}_n^{\mathrm{ST}}$ \\ \hline
\multirow{1}{*}{$0.1$} & $p$  & $4.54\text{e-3}$ & $4.57\text{e-3}$ & $2.67\text{e-3}$ & 2.67\text{e-3} & \multirow{1}{*}{8} & \multirow{1}{*}{8} \\  \hline 
\multirow{1}{*}{$0.2$} & $p$  & $-5.12\text{e-3}$ & $-5.04\text{e-3}$ & $5.23\text{e-3}$ & 5.24\text{e-3} & \multirow{1}{*}{0} & \multirow{1}{*}{0} \\  \hline 
\multirow{1}{*}{$0.3$} & $p$  & $-8.46\text{e-3}$ & $-8.28\text{e-3}$ & $6.61\text{e-3}$ & 6.62\text{e-3} & \multirow{1}{*}{0} & \multirow{1}{*}{0} \\  \hline 
\multirow{1}{*}{$0.4$} & $p$  & $-9.9\text{e-3}$ & $-9.6\text{e-3}$ & $6.85\text{e-3}$ & 6.87\text{e-3} & \multirow{1}{*}{0} & \multirow{1}{*}{0} \\  \hline 
\multirow{1}{*}{$0.5$} & $p$  & $-0.011$ & $-0.011$ & $6.38\text{e-3}$ & 6.39\text{e-3} & \multirow{1}{*}{0} & \multirow{1}{*}{0} \\  \hline 
\multirow{1}{*}{$0.6$} & $p$  & $-0.011$ & $-0.011$ & $5.42\text{e-3}$ & 5.43\text{e-3} & \multirow{1}{*}{0} & \multirow{1}{*}{0} \\  \hline 
\multirow{1}{*}{$0.7$} & $p$  & $-0.011$ & $-0.01$ & $4.03\text{e-3}$ & 4.03\text{e-3} & \multirow{1}{*}{0} & \multirow{1}{*}{0} \\  \hline 
\multirow{1}{*}{$0.8$} & $p$  & $-9.31\text{e-3}$ & $-8.56\text{e-3}$ & $2.39\text{e-3}$ & 2.39\text{e-3} & \multirow{1}{*}{0} & \multirow{1}{*}{0} \\  \hline 
\multirow{1}{*}{$0.9$} & $p$  & $-6.38\text{e-3}$ & $-5.74\text{e-3}$ & $8.55\text{e-4}$ & 8.59\text{e-4} & \multirow{1}{*}{0} & \multirow{1}{*}{0} \\  \hline 
\multirow{1}{*}{$0.95$} & $p$  & $-3.89\text{e-3}$ & $-3.47\text{e-3}$ & $2.88\text{e-4}$ & 2.91\text{e-4} & \multirow{1}{*}{0} & \multirow{1}{*}{0} \\  \hline 
\end{tabular} 
\caption{\protect\label{logarthmic_sim} Simulation results for the $LG(p)$ distribution for $n=50$ and $10{,}000$ repetitions.}
\end{table}

\begin{figure}[ht!]
\centering
\captionsetup{width=.95\textwidth}
\includegraphics[width=7.4cm]{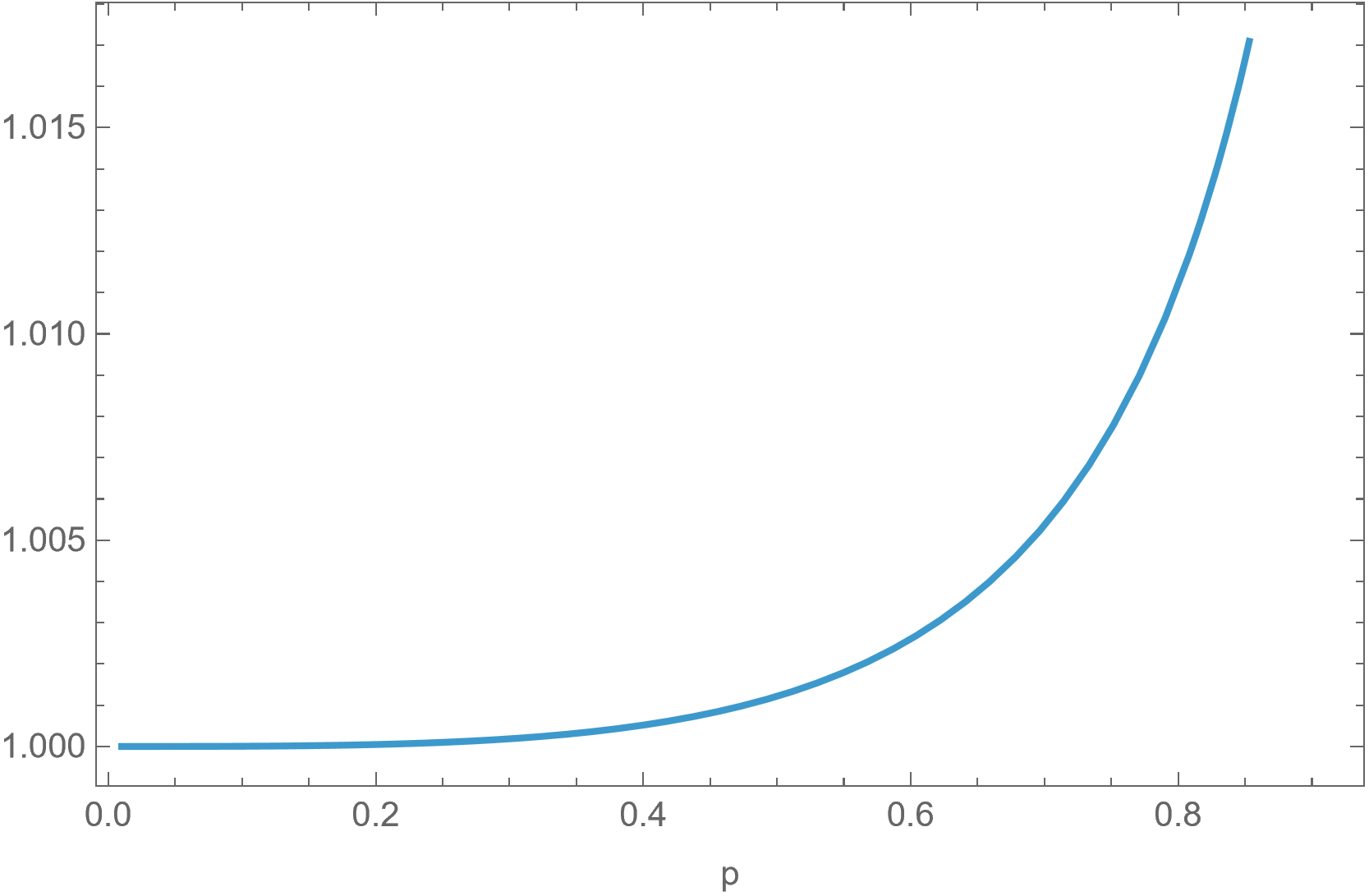}
\caption{\protect\label{fig:logarithmic_variance} \it Relative asymptotic efficiency of the Stein estimator $V^{\mathrm{ST}}(p)/V^{\mathrm{ML}}(p)$ as a function of $p$ for the logarithmic distribution.}
\end{figure}

As already mentioned in the introduction, Stein estimators are suitable for truncated distributions since complicated normalizing constants arising from integration (or summation) over the truncation domain cancel out. More precisely, the Stein operator is the same as in the untruncated case, whereby the function class has to be adjusted such that test functions vanish at the boundary of the truncation domain. This has already been demonstrated in \cite{fischer2025stein} for multivariate truncated continuous probability distributions. Let us revisit Example \ref{example_poisson} and Example \ref{example_binomial} with a truncated domain.

\begin{Example} \label{example_truncated_poisson}
    We consider a truncated Poisson distribution $TP(\lambda,a,b)$ with parameter $\theta=\lambda>0$ on $U=\{a,\ldots,b\}$ where $0 \leq a<b$. The pdf is given by
    \begin{align*}
        p_{\theta}(k)= \bigg( \sum_{i=a}^{b} \frac{\lambda^i}{i!} \bigg)^{-1} \frac{\lambda^k}{k!}.
    \end{align*}
    The Stein operator and estimator read the same as in \eqref{poisson_stein_operator} and \eqref{poisson_stein_estimator}. However, due to the truncation, the function class $\mathscr{F}_{\theta}$ is now such that $f(k)p_{\theta}(k)\tau_{\theta}(k)$ needs to cancel out at $a$ and $b$. In Table \ref{truncpoisson_sim} we report simulation results for the Stein estimator $\hat{\theta}_n^{\mathrm{ST}}$ and MLE $\hat{\theta}_n^{\mathrm{ML}}$. We used the test function
    \begin{align*}
        f(k) = \begin{cases} 0, & \text{if } k=a \text{ or } k=b \\ k, & \text{otherwise} \end{cases}.
    \end{align*}
    Regarding the MLE, maximization of the log-likelihood was performed with the L-BFGS-B option of the R \texttt{optim} function with starting value $\lambda=1$. We observe a similar performance of MLE and Stein estimator, whereby the $\hat{\theta}_n^{\mathrm{ST}}$ seems to be slightly more powerful in terms of bias and $\hat{\theta}_n^{\mathrm{ML}}$ in terms of MSE.
\end{Example}

\begin{table} 
\centering
\begin{tabular}{cc|cc|cc|cc}
 $(\lambda^{\star},a,b)$ & & \multicolumn{2}{c|}{Bias} & \multicolumn{2}{c|}{MSE}  & \multicolumn{2}{c}{NE} \\ \hline
& & $\hat{\theta}_n^{\mathrm{ML}}$ & $\hat{\theta}_n^{\mathrm{ST}}$ & $\hat{\theta}_n^{\mathrm{ML}}$ & $\hat{\theta}_n^{\mathrm{ST}}$ & $\hat{\theta}_n^{\mathrm{ML}}$ & $\hat{\theta}_n^{\mathrm{ST}}$ \\ \hline
\multirow{1}{*}{$(0.1,2,10)$} & $\lambda$  & $-1.53\text{e-3}$ & $0.022$ & $5.76\text{e-3}$ & 4.97\text{e-3} & \multirow{1}{*}{0} & \multirow{1}{*}{19} \\  \hline 
\multirow{1}{*}{$(0.5,0,30)$} & $\lambda$  & $-2.61\text{e-4}$ & $-2.62\text{e-4}$ & $9.83\text{e-3}$ & 9.83\text{e-3} & \multirow{1}{*}{0} & \multirow{1}{*}{0} \\  \hline 
\multirow{1}{*}{$(0.9,6,\infty)$} & $\lambda$  & $0.215$ & $3.58\text{e-3}$ & $0.09$ & 0.111 & \multirow{1}{*}{43} & \multirow{1}{*}{0} \\  \hline 
\multirow{1}{*}{$(1,0,80)$} & $\lambda$  & $-6.34\text{e-4}$ & $-6.34\text{e-4}$ & $0.02$ & 0.02 & \multirow{1}{*}{0} & \multirow{1}{*}{0} \\  \hline 
\multirow{1}{*}{$(1.5,6,85)$} & $\lambda$  & $-0.021$ & $-6.58\text{e-3}$ & $0.15$ & 0.171 & \multirow{1}{*}{0} & \multirow{1}{*}{0} \\  \hline 
\multirow{1}{*}{$(2,2,40)$} & $\lambda$  & $-0.011$ & $-2.4\text{e-3}$ & $0.068$ & 0.076 & \multirow{1}{*}{0} & \multirow{1}{*}{0} \\  \hline 
\multirow{1}{*}{$(2.5,6,90)$} & $\lambda$  & $-0.024$ & $-1.5\text{e-4}$ & $0.205$ & 0.248 & \multirow{1}{*}{0} & \multirow{1}{*}{0} \\  \hline 
\multirow{1}{*}{$(3,0,50)$} & $\lambda$  & $-1.68\text{e-3}$ & $-1.68\text{e-3}$ & $0.06$ & 0.06 & \multirow{1}{*}{0} & \multirow{1}{*}{0} \\  \hline 
\multirow{1}{*}{$(3.5,1,10)$} & $\lambda$  & $-2.9\text{e-3}$ & $-1.11\text{e-3}$ & $0.076$ & 0.077 & \multirow{1}{*}{0} & \multirow{1}{*}{0} \\  \hline 
\multirow{1}{*}{$(4,0,20)$} & $\lambda$  & $-2.07\text{e-3}$ & $-2.07\text{e-3}$ & $0.079$ & 0.079 & \multirow{1}{*}{0} & \multirow{1}{*}{0} \\  \hline 
\end{tabular} 
\caption{\protect\label{truncpoisson_sim} Simulation results for the $TP(\lambda,a,b)$ distribution for $n=50$ and $10{,}000$ repetitions.}
\end{table}

\begin{Example} \label{example_truncated_binomial}
    We consider a truncated binomial distribution $TB(m,p,a,b)$ with parameter $\theta=p>0$ with $0<p<1$ on $U=\{a,\ldots,b\}$ where $0\leq a < b \leq m$. The pdf is given by
    \begin{align*}
        p_{\theta}(k)= \bigg( \sum_{i=a}^{b} \binom{m}{i} p^i (1-p)^{m-i} \bigg)^{-1} \binom{m}{k} p^k (1-p)^{m-k}.
    \end{align*}
    The Stein operator and estimator read the same as in \eqref{binomial_stein_operator} and \eqref{binomial_stein_estimator}. Akin to the previous example we need that $f(k)p_{\theta}(k)\tau_{\theta}(k)$ cancel out at $a$ and $b$. In Table \ref{truncbinom_sim} we report simulation results for the Stein estimator $\hat{\theta}_n^{\mathrm{ST}}$ and MLE $\hat{\theta}_n^{\mathrm{ML}}$. We used the test function
    \begin{align*}
        f(k) = \begin{cases} 0, & \text{if } k=a \text{ or } k=b+1 \\ k, & \text{otherwise} \end{cases}
    \end{align*}
    and computed the MLE by numerical optimization with starting value $p=0.5$. Analogously to Example \ref{example_truncated_poisson}, we observe a small MSE-bias tradeoff.
\end{Example}

\begin{table} 
\centering
\begin{tabular}{cc|cc|cc|cc}
 $(p^{\star},a,b)$ & & \multicolumn{2}{c|}{Bias} & \multicolumn{2}{c|}{MSE}  & \multicolumn{2}{c}{NE} \\  \hline
& & $\hat{\theta}_n^{\mathrm{ML}}$ & $\hat{\theta}_n^{\mathrm{ST}}$ & $\hat{\theta}_n^{\mathrm{ML}}$ & $\hat{\theta}_n^{\mathrm{ST}}$ & $\hat{\theta}_n^{\mathrm{ML}}$ & $\hat{\theta}_n^{\mathrm{ST}}$  \\ \hline
\multirow{1}{*}{$(0.01,10,0,3)$} & $p$  & $-2.13\text{e-5}$ & $2.18\text{e-5}$ & $1.97\text{e-5}$ & 1.91\text{e-5} & \multirow{1}{*}{0} & \multirow{1}{*}{1} \\  \hline 
\multirow{1}{*}{$(0.05,50,0,5)$} & $p$  & $1.26\text{e-4}$ & $6.82\text{e-5}$ & $2.46\text{e-5}$ & 2.67\text{e-5} & \multirow{1}{*}{0} & \multirow{1}{*}{0} \\  \hline 
\multirow{1}{*}{$(0.1,20,2,9)$} & $p$  & $-5.89\text{e-4}$ & $-2.5\text{e-4}$ & $1.58\text{e-4}$ & 1.73\text{e-4} & \multirow{1}{*}{0} & \multirow{1}{*}{0} \\  \hline 
\multirow{1}{*}{$(0.2,20,5,10)$} & $p$  & $-1.21\text{e-3}$ & $-6.04\text{e-4}$ & $4.58\text{e-4}$ & 5.44\text{e-4} & \multirow{1}{*}{0} & \multirow{1}{*}{0} \\  \hline 
\multirow{1}{*}{$(0.25,30,10,20)$} & $p$  & $-1.72\text{e-3}$ & $-8.35\text{e-4}$ & $5.11\text{e-4}$ & 6.89\text{e-4} & \multirow{1}{*}{0} & \multirow{1}{*}{0} \\  \hline 
\multirow{1}{*}{$(0.4,50,2,40)$} & $p$  & $-9.15\text{e-5}$ & $-9.18\text{e-5}$ & $9.48\text{e-5}$ & 9.48\text{e-5} & \multirow{1}{*}{0} & \multirow{1}{*}{0} \\  \hline 
\multirow{1}{*}{$(0.5,50,25,35)$} & $p$  & $-1.15\text{e-3}$ & $-7.38\text{e-4}$ & $2.77\text{e-4}$ & 4.41\text{e-4} & \multirow{1}{*}{0} & \multirow{1}{*}{0} \\  \hline 
\multirow{1}{*}{$(0.75,30,0,25)$} & $p$  & $3.83\text{e-4}$ & $1.61\text{e-4}$ & $1.68\text{e-4}$ & 1.93\text{e-4} & \multirow{1}{*}{0} & \multirow{1}{*}{0} \\  \hline 
\multirow{1}{*}{$(0.8,10,1,8)$} & $p$  & $1.24\text{e-3}$ & $6.68\text{e-4}$ & $5.96\text{e-4}$ & 6.42\text{e-4} & \multirow{1}{*}{0} & \multirow{1}{*}{0} \\  \hline 
\multirow{1}{*}{$(0.99,30,20,29)$} & $p$  & $3.45\text{e-5}$ & $4.36\text{e-6}$ & $1.22\text{e-5}$ & 1.25\text{e-5} & \multirow{1}{*}{0} & \multirow{1}{*}{0} \\  \hline 
\end{tabular} 
\caption{\protect\label{truncbinom_sim} Simulation results for the $TB(m,p,a,b)$ distribution for $n=50$ and $10{,}000$ repetitions.}
\end{table}

\section{Multivariate distributions} \label{section_multivariate}
In this section we extend the notion of Stein estimators for discrete probability distributions to the multivariate setting. We consider probability distribution $\mathbb{P}_{\theta}$ with support $U = \{a_1, \ldots, b_1\} \times \ldots \times \{a_d, \ldots, b_d\} \subset \mathbb{Z}^d$, where $ -\infty \leq a_i < b_i \leq \infty$, $i=1,\ldots , d$ and corresponding probability mass function $p_{\theta}(k)$, $ k \in U$. As a first step, we generalize the discrete Stein operator \eqref{def_stein_op} for multivariate random variables. To this end, let us define the operator $\Delta^{+}f(k) = (\Delta_1^{+}f(k), \ldots,\Delta_d^{+}f(k) ) $, where $\Delta_i^{+}f(k) = f(k_1, \ldots, k_i+1, \ldots , k_d) - f(k) $ for a function $f: \{a_1, \ldots, b_1+1\} \times \ldots \times \{a_d, \ldots, b_d+1\} \rightarrow \mathbb{R}$, where we set $f(x)=0$ if $x \notin U$. For vector-valued $f=(f_1,\ldots,f_d)$ we define $\Delta^+f(x) = (\Delta_1^{+}f_1(k), \ldots,\Delta_d^{+}f_d(k) )^{\top} $. Moreover, we introduce the notation $\Diamond_i^{+}f(k) = f(k_1,\ldots,k_i+1,\ldots,k_d)$ and we can write $\Delta_i^{+}f(k) = \Diamond_i^{+}f(k) - f(k)$. Then, the Stein operator is defined as
\begin{align*}
    \mathcal{A}_{\theta} f(k) =& \frac{\Delta^{+} (f(k) \tau_{\theta}(k) p_{\theta}(k)) }{p_{\theta}(k)} \\
    &= \frac{1 }{p_{\theta}(k)} \Big( \Diamond_1^{+} \big( f(k) p_{\theta}(k) \tau_{\theta}^{(1)}(k) \big), \ldots, \Diamond_d^{+} \big( f(k) p_{\theta}(k) \tau_{\theta}^{(d)}(k) \big) \Big)^{\top}  \\
    & \qquad \qquad \qquad \qquad - \big( f(k) \tau_{\theta}(k), \ldots,  f(k) \tau_{\theta}(k) \big)^{\top}, \qquad k \in U
\end{align*}
with the convention that $p_{\theta}(k)$ and $ \tau_{\theta}(k)$ are equal to some arbitrary number if $k \notin U$ and here $\tau_{\theta}=(\tau_{\theta}^{(1)}, \ldots, \tau_{\theta}^{(d)})$ is now vector-valued and each $\tau_{\theta}^{(i)}$, $i=1,\ldots,d$ is a function $\tau_{\theta}^{(i)}:U \rightarrow \mathbb{R}$. Note that $ \mathcal{A}_{\theta} f:U \rightarrow \mathbb{R}^d$. The corresponding function class $\mathscr{F}_{\theta}$ is
\begin{align*}
    \mathscr{F}_{\theta}=  \bigg\{ f: & \{a_1, \ldots, b_1+1\} \times \ldots \times \{a_d, \ldots, b_d+1\} \rightarrow \mathbb{R} \, \Big\vert  \,  \sum_{k \in U} \Vert  f(k)p_{\theta} (k) \tau_{\theta}(k) \Vert < \infty  \\
    & \text{and } f(k)=0, k=(k_1,\ldots,k_d) \text{ if } k_i=a_i \text{ or } k_i=b_i \text{ for any } i=1,\ldots,d  \bigg\}.
\end{align*}
Again, we let $\mathscr{F} = \cap_{\theta \in \Theta} \mathscr{F}_{\theta} $. The following theorem is the multivariate analog to Theorem \ref{theorem_char_stein_uni}.

\begin{Theorem} \label{theorem_char_stein_mult}
    Let $X \sim \mathbb{P}_{\theta}$. Then, for $\mathcal{A}_{\theta}$ and $\mathscr{F}$ as introduced above,
    \begin{align*}
        \mathbb{E}[\mathcal{A}_{\theta} f(X) ] = 0
    \end{align*}
    for all $f \in \mathscr{F}$.
\end{Theorem}
\begin{proof}
    We consider each component of the Stein operator separately. Let $i \in \{1, \ldots, d\}$. If $-\infty < a_i < b_i < \infty$ then,
    \begin{align*}
        \mathbb{E}[[\mathcal{A}_{\theta} f(X)]_i ] &= \sum_{k \in U} \Delta_i^{+} (f(k) \tau_{\theta}^{(i)}(k) p_{\theta}(k)) \\
        &= \sum_{\substack{a_j \leq k_j \leq b_j, \\ j \neq i }} \sum_{k_i=a_i}^{b_i}  \Delta_i^{+} (f(k) \tau_{\theta}^{(i)}(k) p_{\theta}(k)) \\
        & = \sum_{\substack{a_j \leq k_j \leq b_j, \\ j \neq i }} f(k[i=b_i]) \tau_{\theta}^{(i)}(k[i=b_i]) p_{\theta}(k[i=b_i]) \\
        & \qquad \qquad \qquad \qquad \qquad  -f(k[i=a_i]) \tau_{\theta}^{(i)} (k[i=a_i]) p_{\theta}(k[i=a_i]) \\
        &= 0,
    \end{align*}
    where $k[i=l]=(k_1,\ldots,k_{i-1},l,k_{i+1},\ldots,k_d)$ and $[\cdot]_i$ denotes the $i$th component of a vector. Let us consider the case where $-\infty = a_i < b_i = \infty$. Note that $\sum_{k \in U} \Vert  f(k)p_{\theta} (k) \tau_{\theta}(k) \Vert < \infty$ implies that $\lim_{k_i \rightarrow - \infty} f(k) \tau_{\theta}(k)p_{\theta}(k) = 0 $ and $ \lim_{k_i \rightarrow  \infty} f(k) \tau_{\theta}(k)p_{\theta}(k) = 0 $. Then
    \begin{align*}
        \mathbb{E}[[\mathcal{A}_{\theta} f(X)]_i ] &= \sum_{k \in U} \Delta_i^{+} (f(k) \tau_{\theta}^{(i)}(k) p_{\theta}(k)) \\
        &= \sum_{\substack{a_j \leq k_j \leq b_j, \\ j \neq i }} \sum_{k_i=-\infty}^{\infty} \Delta_i^{+} (f(k) \tau_{\theta}^{(i)}(k) p_{\theta}(k)) \\
        & = \sum_{\substack{a_j \leq k_j \leq b_j, \\ j \neq i }} \lim_{l \rightarrow \infty} f(k[i=l]) \tau_{\theta}^{(i)}(k[i=l]) p_{\theta}(k[i=l]) \\
        & \qquad \qquad \qquad \qquad \qquad  - \lim_{l \rightarrow -\infty} f(k[i=l]) \tau_{\theta}^{(i)} (k[i=l]) p_{\theta}(k[i=l]) \\
        &= 0,
    \end{align*}
    If only $a_i= -\infty$ or $b_i=\infty$, the same argument works. 
\end{proof}
In what follows we present two multivariate examples. We note that for both examples, the dimension of the parameter space $\Theta$ grows with $d$, whereby the Stein estimator is always available in closed-form, also in a high-dimensional setup.
\begin{Example} \label{example_diri_neb_mult}
    For the Dirichlet negative multinomial distribution $DNM(r,\alpha_0,\alpha)$ with $\theta=\alpha$ where $\alpha \in (0, \infty)^d$, $r \in \mathbb{N}_0$, $\alpha_0>0$ and $U= \mathbb{N}_0^d$ we have 
    \begin{align*}
        p_{\theta}(k) = \frac{B\big(r +\sum_{i=1}^d k_i, \alpha_0 +\sum_{i=1}^d \alpha_i\big)}{B(r,\alpha_0)} \prod_{i=1}^d \frac{\Gamma(k_i+\alpha_i)}{k_i! \Gamma(\alpha_i)}.
    \end{align*}
    Let $\tau_{\theta}^{(i)}(k) =k_i \big( \sum_{j=1}^d k_j -1 + r + \alpha_0 +\sum_{j=1}^d \alpha_j  \big)  $ for $i = 1\ldots, d$ and we get for the $i$th component of the Stein operator
    \begin{align*}
        [\mathcal{A}_{\theta} f(k)]_i =  \bigg(\sum_{j=1}^d k_j +r \bigg) (k_i+\alpha_i) \Diamond_i^{+} f(k) - k_i \bigg( \sum_{j=1}^d k_j - 1 + r + \alpha_0 +\sum_{j=1}^d \alpha_j  \bigg) f(k) 
    \end{align*}
    For known $r$ and $\alpha_0$ and an i.i.d.\ sample $X_1,\ldots,X_n \sim DNM(r,\alpha_0,\alpha^{\star})$ with $X_i = (X_i^{(1)}, \ldots, X_i^{(d)} )^{\top}$ this gives the Stein estimator $\hat{\alpha}_n = A_n^{-1} b_n$
    where
    \begin{align*}
        (A_n)_{i,j} = \mathbbm{1}{\{i=j\}} \bigg( \sum_{l=1}^d \overline{ X^{(l)}\Diamond_i^{+} f(X) } + r \overline{\Diamond_i^{+} f(X)} \bigg) -\overline{X^{(i)}f(X)}
    \end{align*}
    and 
    \begin{align*}
        (b_n)_{i} = \sum_{l=1}^d \overline{X^{(l)} X^{(i)} f(X) } + ( r+ \alpha_0 -1) \overline{X^{(i)} f(X) } - \sum_{l=1}^d \overline{ X^{(l)} X^{(i)} \Diamond_i^{+} f(X) } - r \overline{ X^{(i)} \Diamond_i^{+} f(X)},
    \end{align*}
    where $1 \leq i,j \leq d$. Simulation results can be found in Table \ref{dirinegmult_sim} where we consider the Stein estimator $\hat{\theta}_n^{\mathrm{ST}}$, the moment estimator $\hat{\theta}_n^{\mathrm{MO}}$ and the MLE $\hat{\theta}_n^{\mathrm{ML}}$. The moment estimator is only valid for $\alpha_0>1$ and is given by
    \begin{align*}
        \hat{\alpha}_n^{\mathrm{MO}} = \frac{(\alpha_0-1) \overline{X}}{r}.
    \end{align*}
    For the Stein estimator we used the test function
    \begin{align*}
        f(k) = \begin{cases} 0, & \text{if }  \min\{k_1, \ldots, k_d\}=0 \\ \big( \sum_{i=1}^{d} k_i \big)^{-1}, & \text{otherwise} \end{cases}.
    \end{align*}
    Regarding the MLE we used the Nelder-Mead algorithm as implemented in \texttt{optim} for numerical optimization with starting value $(1,\ldots,1)$ and considered the estimate as non-eligible if the runtime of the optimization algorithm exceeded $10$ seconds. As it can be observed in Table \ref{dirinegmult_sim}, for parameter constellation where the MLE and moment estimator produced eligible estimates, the MLE seems to achieve the smallest MSE and the moment estimator together with the MLE the smallest bias, with the Stein estimator still performing reasonably well. However, for several parameter values, the MLE did not return a single eligible estimate (this seems to be the case especially when $\alpha_0<1$) while the Stein estimator still produced consistent results. The high values in column NE for the MLE were mostly a consequence of a violation of the runtime constraint.
\end{Example}

\begin{table} 
\centering
\begin{tabular}{cc|ccc|ccc|ccc}
 $(r,\alpha_0,\alpha)$ & & \multicolumn{3}{c|}{Bias} & \multicolumn{3}{c|}{MSE} & \multicolumn{3}{c}{NE} \\ \hline
& & $\hat{\theta}_n^{\mathrm{ML}}$ & $\hat{\theta}_n^{\mathrm{MO}}$ & $\hat{\theta}_n^{\mathrm{ST}}$ & $\hat{\theta}_n^{\mathrm{ML}}$ & $\hat{\theta}_n^{\mathrm{MO}}$ & $\hat{\theta}_n^{\mathrm{ST}}$ & $\hat{\theta}_n^{\mathrm{ML}}$ & $\hat{\theta}_n^{\mathrm{MO}}$ & $\hat{\theta}_n^{\mathrm{ST}}$ \\ \hline
\multirow{3}{*}{$(5,0.5,(2,2,2))$} & $\alpha_1$  & -- & -- & 0.026 & -- & -- & 0.056 & \multirow{3}{*}{100} & \multirow{3}{*}{100} & \multirow{3}{*}{0} \\ & $\alpha_2$  & -- & -- & 0.026 & -- & -- & 0.057\\ & $\alpha_{3}$ & -- & -- & 0.027 & -- & -- & 0.057\\ \hline 
\multirow{3}{*}{$(5,2,(2,2,2))$} & $\alpha_1$  & -3.16\text{e-3} & 9.43\text{e-3} & 0.011 & 0.021 & 0.47 & 0.034 & \multirow{3}{*}{69} & \multirow{3}{*}{0} & \multirow{3}{*}{0} \\ & $\alpha_2$  & -2.82\text{e-3} & -7.5\text{e-4} & 9.83\text{e-3} & 0.021 & 0.23 & 0.034\\ & $\alpha_{3}$ & 7.3\text{e-4} & -1.3\text{e-3} & 0.012 & 0.021 & 0.221 & 0.034\\ \hline 
\multirow{3}{*}{$(10,3,(1,1,1))$} & $\alpha_1$  & 3.05\text{e-3} & 1.62\text{e-3} & 6.01\text{e-3} & 5.98\text{e-3} & 0.018 & 0.012 & \multirow{3}{*}{11} & \multirow{3}{*}{0} & \multirow{3}{*}{0} \\ & $\alpha_2$  & 3.29\text{e-3} & 5.07\text{e-4} & 5.56\text{e-3} & 5.98\text{e-3} & 0.018 & 0.012\\ & $\alpha_{3}$ & 2.43\text{e-3} & -1.35\text{e-4} & 6.35\text{e-3} & 6.09\text{e-3} & 0.018 & 0.012\\ \hline 
\multirow{3}{*}{$(2,0.2,(1,0.5,3))$} & $\alpha_1$  & -- & -- & 0.026 & -- & -- & 0.04 & \multirow{3}{*}{100} & \multirow{3}{*}{100} & \multirow{3}{*}{13} \\ & $\alpha_2$  & -- & -- & 0.011 & -- & -- & 9.75\text{e-3}\\ & $\alpha_{3}$ & -- & -- & 0.077 & -- & -- & 0.32\\ \hline 
\multirow{3}{*}{$(1,4,(1,1,1))$} & $\alpha_1$  & 6.85\text{e-3} & 2.27\text{e-3} & 0.103 & 0.029 & 0.04 & 0.382 & \multirow{3}{*}{0} & \multirow{3}{*}{0} & \multirow{3}{*}{0} \\ & $\alpha_2$  & 5.71\text{e-3} & 4.6\text{e-4} & 0.104 & 0.029 & 0.04 & 0.389\\ & $\alpha_{3}$ & 7.52\text{e-3} & 2.42\text{e-3} & 0.11 & 0.029 & 0.041 & 0.403\\ \hline 
\multirow{3}{*}{$(3,0.8,(0.5,0.5,2))$} & $\alpha_1$  & -- & -- & 7.02\text{e-3} & -- & -- & 7.14\text{e-3} & \multirow{3}{*}{100} & \multirow{3}{*}{100} & \multirow{3}{*}{0} \\ & $\alpha_2$  & -- & -- & 6.96\text{e-3} & -- & -- & 7.15\text{e-3}\\ & $\alpha_{3}$ & -- & -- & 0.031 & -- & -- & 0.083\\ \hline 
\multirow{3}{*}{$(4,5,(1,2,1))$} & $\alpha_1$  & 3.3\text{e-3} & 2.31\text{e-4} & 0.014 & 0.011 & 0.016 & 0.041 & \multirow{3}{*}{0} & \multirow{3}{*}{0} & \multirow{3}{*}{0} \\ & $\alpha_2$  & 6.15\text{e-3} & 10\text{e-4} & 0.03 & 0.028 & 0.04 & 0.129\\ & $\alpha_{3}$ & 3.6\text{e-3} & 1.64\text{e-3} & 0.013 & 0.012 & 0.017 & 0.042\\ \hline 
\multirow{3}{*}{$(2,2,(1,1,0.5))$} & $\alpha_1$  & 5.27\text{e-3} & 7.11\text{e-3} & 0.028 & 0.012 & 0.099 & 0.065 & \multirow{3}{*}{8} & \multirow{3}{*}{0} & \multirow{3}{*}{0} \\ & $\alpha_2$  & 6.26\text{e-3} & 8.05\text{e-3} & 0.031 & 0.012 & 0.1 & 0.066\\ & $\alpha_{3}$ & 2.05\text{e-3} & 2.17\text{e-3} & 0.013 & 4.41\text{e-3} & 0.029 & 0.018\\ \hline 
\multirow{4}{*}{$(2,0.5,(1,1,1,1))$} & $\alpha_1$  & -- & -- & 0.022 & -- & -- & 0.03 & \multirow{4}{*}{100} & \multirow{4}{*}{100} & \multirow{4}{*}{0} \\ & $\alpha_2$  & -- & -- & 0.022 & -- & -- & 0.03\\ & $\alpha_{3}$ & -- & -- & 0.022 & -- & -- & 0.03\\ & $\alpha_{4}$ & -- & -- & 0.023 & -- & -- & 0.03\\ \hline 
\multirow{4}{*}{$(8,3,(2,2,2,4))$} & $\alpha_1$  & -5.12\text{e-3} & 1.22\text{e-3} & 8.55\text{e-3} & 0.023 & 0.047 & 0.026 & \multirow{4}{*}{41} & \multirow{4}{*}{0} & \multirow{4}{*}{0} \\ & $\alpha_2$  & -6.92\text{e-3} & -1.06\text{e-3} & 7.57\text{e-3} & 0.024 & 0.047 & 0.026\\ & $\alpha_{3}$ & -7.96\text{e-3} & -6.94\text{e-4} & 5.87\text{e-3} & 0.024 & 0.049 & 0.027\\ & $\alpha_{4}$ & -0.034 & -2.98\text{e-3} & 0.013 & 0.109 & 0.145 & 0.078\\ \hline 
\end{tabular} 
\caption{\protect\label{dirinegmult_sim} Simulation results for the $DNM(r,\alpha_0,\alpha)$ distribution for $n=200$ and $10{,}000$ repetitions.}
\end{table}

\begin{Example} \label{example_truncated_multneg}
    For the negative multinomial distribution $NM(r,p)$ with $\theta=p$ where $p \in (0,1)^d$ such that $p_0 = 1-p_1- \ldots p_d > 0$ and $r > 0$ on $U= \mathbb{N}_0^d$ we have 
    \begin{align*}
        p_{\theta}(k) = \Gamma\bigg( r + \sum_{i=1}^d k_i \bigg) \frac{p_0^r}{\Gamma(r)} \prod_{i=1}^d \frac{p_i^{k_i}}{k_i!}.
    \end{align*}
    Let $\tau_{\theta}^{(i)}(k) = k_i $ for $i = 1\ldots, d$ and we get for the $i$th component of the Stein operator
    \begin{align*}
        [ \mathcal{A}_{\theta} f(k) ]_i =  \bigg(\sum_{j=1}^d k_j +r \bigg) p_i \Diamond_i^{+} f(k) - k_i f(k).
    \end{align*}
    For known $r$ and an i.i.d.\ sample $X_1,\ldots,X_n \sim NM(r,p^{\star})$ with $X_i = (X_i^{(1)}, \ldots, X_i^{(d)} )^{\top}$ this gives the Stein estimator 
    \begin{align*}
        (\hat{p}_n)_{i} = \overline{X^{(i)}f(X)} \bigg( \sum_{l=1}^d \overline{ X^{(l)}\Diamond_i^{+} f(X) } + r \overline{\Diamond_i^{+} f(X)} \bigg)^{-1},
    \end{align*}
    where $1 \leq i \leq d$. Here we consider the distribution in a truncated setting, i.e. $TNM(r,p,a,b)$ with $r > 0$, $a=(a_1,\ldots,a_d)$, $b=(b_1,\ldots,b_d)$, $\theta=p$ where $p \in (0,1)^d$ such that $p_0 = 1-p_1- \ldots p_d > 0$ and $r > 0$ with support $U= \{a_1,\dots,b_1\} \times \ldots \times \{a_d,\dots,b_d\}$. The probability mass function is given by 
    \begin{align*}
        p_{\theta}(k) = \bigg(\sum_{k' \in U} \Gamma\bigg( r + \sum_{i=1}^d k_i' \bigg)  \prod_{i=1}^d \frac{p_i^{k_i'}}{k_i'!} \bigg)^{-1} \Gamma\bigg( r + \sum_{i=1}^d k_i \bigg)  \prod_{i=1}^d \frac{p_i^{k_i}}{k_i!}.
    \end{align*}  
    Stein operator and estimator are defined as above in untruncated case. Simulation results can be found in Table \ref{truncnegmult_sim} for the Stein estimator $\hat{\theta}_n^{\mathrm{ST}}$ with test function
    \begin{align*}
        f(k) = \begin{cases} 0, & \text{if } k_i=a_i \text{ or } k_i=b_i \text{ for any } i=1,\ldots,d \\ \sum_{i=1}^{d} k_i, & \text{otherwise} \end{cases}
    \end{align*}
    and the MLE $\hat{\theta}_n^{\mathrm{ML}}$ for different parameter constellations and truncation domains. For the computation of the MLE, we used again the Nelder-Mead algorithm with a runtime constraint of $10$ seconds, with starting value $(1/(d+1),\dots,1/(d+1))$. Both estimators performed well whenever estimates were eligible. Nonetheless, for many parameter constellations, the numerical optimization of the MLE encountered substantial difficulties in producing a valid estimate within the prescribed runtime limit while the Stein estimator returned an eligible value for most Monte Carlo repetitions.
\end{Example}

\begin{table} 
\centering
\begin{tabular}{cc|cc|cc|cc}
 $(r,p,a,b)$ & & \multicolumn{2}{c|}{Bias} & \multicolumn{2}{c|}{MSE} & \multicolumn{2}{c}{NE} \\ \hline
& & $\hat{\theta}_n^{\mathrm{ML}}$ & $\hat{\theta}_n^{\mathrm{ST}}$ & $\hat{\theta}_n^{\mathrm{ML}}$ & $\hat{\theta}_n^{\mathrm{ST}}$ & $\hat{\theta}_n^{\mathrm{ML}}$ & $\hat{\theta}_n^{\mathrm{ST}}$ \\ \hline
\multirowcell{3}{$(5,(0.3,0.1,0.2) $, \\ $(0,0,0),(10,10,10))$} & $p_1$  & -- & 3.1\text{e-4} & -- & 7.35\text{e-4} & \multirow{3}{*}{100} & \multirow{3}{*}{0} \\ & $p_2$  & -- & -2.07\text{e-4} & -- & 9.95\text{e-5}\\ & $p_{3}$ & -- & -1.96\text{e-4} & -- & 2.9\text{e-4}\\ \hline 
\multirowcell{3}{$(5,(0.1,0.1,0.1) $, \\ $(1,1,2),(5,5,5))$} & $p_1$  & 4.04\text{e-6} & 1.24\text{e-3} & 1.33\text{e-4} & 1.32\text{e-3} & \multirow{3}{*}{5} & \multirow{3}{*}{0} \\ & $p_2$  & 5.85\text{e-5} & 7.39\text{e-4} & 1.35\text{e-4} & 1.3\text{e-3}\\ & $p_{3}$ & -3.22\text{e-4} & 2.14\text{e-3} & 2.14\text{e-4} & 1.57\text{e-3}\\ \hline 
\multirowcell{3}{$(7,(0.2,0.1,0.1) $, \\ $(0,0,0),(5,5,5))$} & $p_1$  & -9.84\text{e-3} & 2.33\text{e-3} & 1.97\text{e-4} & 1.26\text{e-3} & \multirow{3}{*}{99} & \multirow{3}{*}{0} \\ & $p_2$  & 4.06\text{e-3} & 3.24\text{e-5} & 5.06\text{e-5} & 2.35\text{e-4}\\ & $p_{3}$ & 5\text{e-3} & 1.89\text{e-4} & 5.09\text{e-5} & 2.33\text{e-4}\\ \hline 
\multirowcell{3}{$(2,(0.8,0.1,0.05) $, \\ $(0,0,2),(7,8,5))$} & $p_1$  & -- & -0.139 & -- & 0.05 & \multirow{3}{*}{100} & \multirow{3}{*}{33} \\ & $p_2$  & -- & -5.82\text{e-4} & -- & 8.77\text{e-4}\\ & $p_{3}$ & -- & 3.97\text{e-4} & -- & 3.34\text{e-4}\\ \hline 
\multirowcell{3}{$(5,(0.5,0.1,0.1) $, \\ $(2,1,1),(5,5,5))$} & $p_1$  & -1.86\text{e-5} & 0.01 & 2.38\text{e-3} & 0.032 & \multirow{3}{*}{23} & \multirow{3}{*}{8} \\ & $p_2$  & -2.92\text{e-4} & 8.35\text{e-4} & 1.27\text{e-4} & 7\text{e-4}\\ & $p_{3}$ & -4.12\text{e-4} & 4.08\text{e-4} & 1.29\text{e-4} & 6.73\text{e-4}\\ \hline 
\multirowcell{3}{$(3,(0.3,0.3,0.3) $, \\ $(0,0,0),(10,5,5))$} & $p_1$  & -- & -2.95\text{e-4} & -- & 6.33\text{e-4} & \multirow{3}{*}{100} & \multirow{3}{*}{0} \\ & $p_2$  & -- & 4.52\text{e-3} & -- & 2.41\text{e-3}\\ & $p_{3}$ & -- & 4.09\text{e-3} & -- & 2.38\text{e-3}\\ \hline 
\multirowcell{3}{$(5,(0.2,0.3,0.4) $, \\ $(1,1,1),(9,9,9))$} & $p_1$  & -- & 3.32\text{e-4} & -- & 2.65\text{e-4} & \multirow{3}{*}{100} & \multirow{3}{*}{0} \\ & $p_2$  & -- & 1.16\text{e-3} & -- & 8.94\text{e-4}\\ & $p_{3}$ & -- & 4.07\text{e-3} & -- & 2.23\text{e-3}\\ \hline 
\multirowcell{3}{$(1,(0.1,0.1,0.4) $, \\ $(0,0,0),(10,10,7))$} & $p_1$  & -- & -3.85\text{e-3} & -- & 1.39\text{e-3} & \multirow{3}{*}{100} & \multirow{3}{*}{3} \\ & $p_2$  & -- & -3.31\text{e-3} & -- & 1.4\text{e-3}\\ & $p_{3}$ & -- & -0.014 & -- & 0.024\\ \hline 
\multirowcell{3}{$(1,(0.4,0.5,0.05) $, \\ $(2,3,2),(10,10,10))$} & $p_1$  & -- & 5.78\text{e-3} & -- & 5.71\text{e-3} & \multirow{3}{*}{100} & \multirow{3}{*}{0} \\ & $p_2$  & -- & 0.012 & -- & 0.011\\ & $p_{3}$ & -- & 1.39\text{e-4} & -- & 1.14\text{e-4}\\ \hline 
\multirowcell{4}{$(1,(0.4,0.05,0.2,0.01) $, \\ $(1,1,0,0),(10,11,11,12))$} & $p_1$  & -- & -0.05 & -- & 0.013 & \multirow{4}{*}{100} & \multirow{4}{*}{46} \\ & $p_2$  & -- & 0.03 & -- & 2.86\text{e-3}\\ & $p_{3}$ & -- & -0.011 & -- & 6.81\text{e-3}\\ & $p_{4}$ & -- & 6.65\text{e-3} & -- & 1.38\text{e-4}\\ \hline 
\end{tabular}
\caption{\protect\label{truncnegmult_sim} Simulation results for the $TNM(r,p,a,b)$ distribution for $n=100$ and $10{,}000$ repetitions.}
\end{table}

\section{Unknown truncation domain} \label{section_unknwon_truncation}
In this section we investigate truncated distributions from a more theoretical point of view and show that the asymptotic variance is not affected if the domain of the random variable has to be estimated. We state the results for univariate distributions $\, _a^b\mathbb{P}_{\theta}$ with support $U=\{a,\ldots,b\}$ where $-\infty < a<b < \infty$. However, a generalization of the results presented in this section to multivariate distributions with a rectangular support, i.e. $U=\{a_1,\ldots,b_1\} \times \ldots, \times \{a_d,\ldots,b_d\}$ is straightforward. Let $X_1,\ldots,X_n \sim \, _{a^{\star}}^{b^{\star}}\mathbb{P}_{\theta^{\star}}$ be an i.i.d.\ sample on a common probability space $(\Omega,\mathcal{F},\mathbb{P})$. Here we wrote again $\theta^{\star}$, $a^{\star}$ and $b^{\star}$ for the true parameters. As per $a^{\star}$ and $b^{\star}$ we employ the standard estimators
\begin{align*}
    \hat{a}_n= \min\{X_1,\ldots,X_n \}, \qquad \hat{b}_n= \max\{X_1,\ldots,X_n \}.
\end{align*}
We have the following asymptotic result for $\hat{a}_n$ and $\hat{b}_n$.
\begin{Lemma} \label{lemma_conv_maxmin}
    We have for any sequence $q_n$, $n\in \mathbb{N}_0$ that
    \begin{align*}
        q_n (\hat{a}_n-a^{\star}) \xrightarrow{\text{a.s.}} 0 \qquad \text{and} \qquad q_n (\hat{b}_n-b^{\star}) \xrightarrow{\text{a.s.}} 0
    \end{align*}
    as $n \rightarrow \infty$, where $\xrightarrow{\text{a.s.}}$ denotes convergence almost surely.
\end{Lemma}
\begin{proof}
    We only do the computations for $\hat{a}_n$. We have
    \begin{align*}
        \mathbb{P}(\vert q_n (\hat{a}_n-a^{\star}) \vert >0) &= \mathbb{P}(\{X_1 \neq a^{\star} \} \cap \ldots \cap \{X_n \neq a^{\star} \}) \\
        &= (1-p_{\theta}(a^{\star}))^n \rightarrow 0,
    \end{align*}
    as $n \rightarrow \infty$ and therefore $q_n (\hat{a}_n-a^{\star}) \xrightarrow{\mathbb{P}} 0$ as $n \rightarrow \infty$. Since $\hat{a}_n$ and $\hat{b}_n$ take values in a finite set, almost sure convergence follows.
\end{proof}
The consistency of $\hat{a}_n$ and $\hat{b}_n$ is an immediate consequence of the previous lemma. \par

Let us move on to the estimation of $\theta^{\star}$. Note that (in contrast to the Stein operator \eqref{def_stein_op}) the test functions do depend on $a$ and $b$ as to why we write $(f_{a}^{b})_1,\ldots,(f_{a}^{b})_q$ and $\, _{a}^{b}\mathscr{F}_{\theta}$ as well as $\, _{a}^{b}\mathscr{F} = \cap_{\theta \in \Theta} \, _{a}^{b}\mathscr{F}_{\theta} $. With our estimates $\hat{a}_n$ and $\hat{b}_n$ at hand we choose test functions $(f_{\hat{a}_n}^{\hat{b}_n})_1,\ldots,(f_{\hat{a}_n}^{\hat{b}_n})_q \in \, _{\hat{a}_n}^{\hat{b}_n}\mathscr{F}$ and define the Stein estimator $\hat{\theta}_n$ as the solution to
\begin{gather*} 
    \frac{1}{n} \sum_{i=1}^n \mathcal{A}_{\theta}(f_{\hat{a}_n}^{\hat{b}_n})_1(X_i) = 0, \\
    \vdots \\
    \frac{1}{n} \sum_{i=1}^n \mathcal{A}_{\theta}(f_{\hat{a}_n}^{\hat{b}_n})_q(X_i) = 0.
\end{gather*}
In the next theorem, we show that $\hat{\theta}_n$ as defined above exists and is measurable with probability converging to $1$. For the latter purpose we need some technical assumptions which resemble a lot the ones made in \cite[Assumption 2.2]{ebner2025stein}.

\begin{Assumption} \label{assumptions_unknown_trunc} \leavevmode
\begin{description} 
\item[(a)] Let $X \sim \, _{a^{\star}}^{b^{\star}}\mathbb{P}_{\theta^{\star}}$ and $\theta \in \Theta$.  Then $f_{a^{\star}}^{b^{\star}} \in \, _{a^{\star}}^{b^{\star}}\mathscr{F}$, and $\mathbb{E}[\mathcal{A}_{\theta}f_{a^{\star}}^{b^{\star}}(X)]=0$ if and only if $\theta=\theta^{\star}$. Moreover suppose that $\mathbb{E}[\Vert M_{a'}^{b'}(X) \Vert^2 ] < \infty$ for all $a^{\star} \leq a' \leq b' \leq b^{\star}$.
\item[(b)] For $\tilde{q} \geq q$, we can write $\mathcal{A}_{\theta}f_a^b(k) =M_a^b(k)g(\theta)$ for some measurable $q \times \tilde{q}$ matrix $M_a^b$ and a continuously differentiable function $g=(g_1,\ldots, g_{\tilde{q}})^\top:\Theta \rightarrow \mathbb{R}^q $ for all $\theta \in \Theta$, $k \in \{a, \ldots, b\}$. Moreover, we assume that $\mathbb{E}[M_{a^{\star}}^{b^{\star}}(X)]\frac{\partial}{\partial \theta} g(\theta) \vert_{\theta=\theta^{\star}}$, where $X \sim \, _{a^{\star}}^{b^{\star}}\mathbb{P}_{\theta^{\star}}$ is invertible.
\end{description}
\end{Assumption}

We then have the following theorem.
\begin{Theorem} \label{theorem_consistency_unknown_trunc}
   Suppose  Assumptions \ref{assumptions_unknown_trunc}(a)--(b)
   are fulfilled. The probability that $\hat{\theta}_n$ exists and is measurable converges to $1$ as $n \rightarrow \infty$.
\end{Theorem}
\begin{proof}
Let $X \sim \, _{a^{\star}}^{b^{\star}}\mathbb{P}_{\theta^{\star}}$. Define the function $F(M,\theta)=Mg(\theta)$, where $M \in \mathbb{R}^{q \times \tilde{q}}$ and $g$ is defined as in Assumption \ref{assumptions_unknown_trunc}(b). Then $F$ is continuously differentiable on $\mathbb{R}^{q \times \tilde{q}} \times \Theta$ and we have $F(\mathbb{E}[M_{a^{\star}}^{b^{\star}}(X)],\theta^{\star})=0$ by Assumption \ref{assumptions_unknown_trunc}(a), where $M_a^b(x)$ is defined as in Assumption \ref{assumptions_unknown_trunc}(b). The implicit function theorem implies that there are neighborhoods $U \subset \mathbb{R}^{q \times \tilde{q}},$ $V \subset \mathbb{R}^q$ of $\mathbb{E}[M_{a^{\star}}^{b^{\star}}(X)]$ and $\theta^{\star}$ such that there is a continuously differentiable function $h:U \rightarrow V$ with $F(M,h(M))=0$ for all $M \in U$. Now, for any $\epsilon >0$,
\begin{align*}
    &\mathbb{P} \bigg(\bigg\Vert\frac{1}{n}\sum_{i=1}^n M_{\hat{a}_n}^{\hat{b}_n}(X_i) - \mathbb{E}[M_{\hat{a}_n}^{\hat{b}_n}(X)] \bigg\Vert > \epsilon \bigg) \\
    &\leq \mathbb{P} \bigg(\max_{a^{\star} \leq a' \leq b' \leq b^{\star}} \bigg\Vert \frac{1}{n} \sum_{i=1}^n M_{a'}^{b'}(X_i) - \mathbb{E}[M_{a'}^{b'}(X)] \bigg\Vert > \epsilon  \bigg),
\end{align*}
where the expectation in the first term is with respect to $X$ and the last term converges to $0$ as we have a maximum over a finite set. Moreover, we know that $\mathbb{E}[M_{\hat{a}_n}^{\hat{b}_n}(X)] \xrightarrow{\mathbb{P}} \mathbb{E}[M_{a^{\star}}^{b^{\star}}(X)]$, where the first expectation is taken with respect to $X$. Hence, we can conclude that $n^{-1}\sum_{i=1}^n M_{\hat{a}_n}^{\hat{b}_n}(X_i)$ is converging to $\mathbb{E}[M_{a^{\star}}^{b^{\star}}(X)]$ in probability. Thus, by defining $A_n=\big\{ n^{-1} \sum_{i=1}^n M_{\hat{a}_n}^{\hat{b}_n}(X_i) \in U \big\}$, we have $\mathbb{P}( A_n) \rightarrow 1$ as $n \rightarrow \infty$. Therefore, we have shown existence and  measurability for each $\omega \in A_n$ with $\mathbb{P}(A_n) \rightarrow 1$, since we have  $\hat{\theta}_n=h\big(n^{-1} \sum_{i=1}^n M_{\hat{a}_n}^{\hat{b}_n}(X_i)\big)$. 
\end{proof}

From the proof of Theorem \ref{theorem_consistency_unknown_trunc} and the strong law of large numbers it also follows that the sequence $\{\hat{\theta}_n \, \vert \, n \in \mathbb{N} \}$ is consistent for $\theta^{\star}$.

We make use of the following discrete Taylor expansion formula \cite[Theorem 1.8.5]{agarwal2000difference}. For any function $g$ defined on $\{\tilde{k},\tilde{k}+1, \ldots  \}$ we have
\begin{align} \label{discrete_taylor_lower}
    g(k)=\sum_{i=0}^{n-1} \frac{(k-\tilde{k})^{(i)}}{i!} (\Delta^+)^i g(\tilde{k}) + \frac{1}{(n-1)!} \sum_{i=\tilde{k}}^{k-n} (k-i-1)^{(n-1)} (\Delta^+)^n g(i),
\end{align}
where $(\Delta^+)^n g(k)=\Delta^+ ((\Delta^+)^{n-1} g(k))$ with the convention $(\Delta^+)^0 g(k)=g(k)$ and $(k)^{(m)}= \prod_{i=0}^{m-1} (k-i) $ for all $k \geq a$ with the convention $(k)^{(0)}=1$ and $n \geq 1$. By rearranging the elements of $\{\tilde{k},\tilde{k}+1,\ldots\}$ we have that for a function $g$ defined on $\{\ldots,\tilde{k}-1,\tilde{k}\}$ that
\begin{align} \label{discrete_taylor_upper}
    g(k)=\sum_{i=0}^{n-1} (-1)^{i} \frac{(\tilde{k}-k)^{(i)}}{i!} (\Delta^-)^i g(\tilde{k}) + \frac{(-1)^{n}}{(n-1)!} \sum_{i=k+n}^{\tilde{k}} (i-k-1)^{(n-1)} (\Delta^-)^n g(i).
\end{align}
We note that the sums in formulas \eqref{discrete_taylor_lower} and \eqref{discrete_taylor_upper} are equal to $0$ when the upper bound for the iterator is smaller than the lower bound. Now we are able to show that the asymptotic variance of the Stein estimator does not change when $a^{\star}$ and $b^{\star}$ are estimated.
\begin{Theorem} \label{theorem_asymnorm_unknown_trunc}
    Let $X \sim \, _{a^{\star}}^{b^{\star}}\mathbb{P}_{\theta^{\star}}$. Suppose that Assumption \ref{assumptions_unknown_trunc}(a)--(b) hold and that
    \begin{align*}
        \bigg\Vert \frac{\partial^2}{\partial \theta_j \partial \theta_l} \mathcal{A}_{\theta}f_{a'}^{b'} (k) \Big\vert_{\theta=\theta'} \bigg\Vert \leq C
    \end{align*}
    for all $\theta' \in \Theta$, $a^{\star} \leq k \leq b^{\star}$, $a^{\star} \leq a' \leq b' \leq b^{\star}$ and $1 \leq j,l \leq q$ for some $C>0$. Then we have that
    \begin{align*}
        \sqrt{n} (\hat{\theta}_n - \theta^{\star} ) \xrightarrow{D} N(0,\Sigma),
    \end{align*}
    where
    \begin{align*}
        \Sigma =  \mathbb{E}\bigg[ \frac{\partial}{\partial \theta} \mathcal{A}_{\theta} f_{a^{\star}}^{b^{\star}} \Big\vert_{\theta = \theta^{\star}} \bigg]^{-1} \mathbb{E}\bigg[ \big( \mathcal{A}_{\theta^{\star}} f_{a^{\star}}^{b^{\star}}(X) \big)^2 \bigg] \mathbb{E}\bigg[ \frac{\partial}{\partial \theta} \mathcal{A}_{\theta} f_{a^{\star}}^{b^{\star}}(X) \Big\vert_{\theta = \theta^{\star}} \bigg]^{-\top}.
    \end{align*}
\end{Theorem}
\begin{proof}
A multivariate Taylor expansion gives
\begin{align*}
    0=\frac{1}{n} \sum_{i=1}^n \mathcal{A}_{\hat{\theta}_n}f_{\hat{a}_n}^{\hat{b}_n} (X_i) =& \frac{1}{n} \sum_{i=1}^n \mathcal{A}_{\theta^{\star}}f_{\hat{a}_n}^{\hat{b}_n} (X_i) + \bigg( \frac{1}{n} \sum_{i=1}^n \frac{\partial}{\partial \theta} \mathcal{A}_{\theta}f_{\hat{a}_n}^{\hat{b}_n} (X_i) \Big\vert_{\theta=\theta^{\star}} \bigg) (\hat{\theta}_n - \theta^{\star} ) \\
    & + A_n(\hat{\theta}_n,\hat{a}_n,\hat{b}_n) (\hat{\theta}_n - \theta^{\star} ),
\end{align*}
where 
\begin{align*}
    [A_n(\hat{\theta}_n,\hat{a}_n,\hat{b}_n)]_{k,l}= \sum_{j=1}^n \int_0^1 (1-t)  \frac{1}{n} \sum_{i=1}^n \frac{\partial^2}{\partial \theta_j \partial \theta_l} \big[ \mathcal{A}_{\theta}f_{\hat{a}_n}^{\hat{b}_n} (X_i) \big]_{k} \Big\vert_{\theta=\theta^{\star}+t(\hat{\theta}_n - \theta^{\star})}  dt \big[\hat{\theta}_n - \theta^{\star} \big]_{j},
\end{align*}
where $[\cdot]_i$ denotes the $i$th component of a vector or matrix. Rearranging yields
\begin{align*}
    \sqrt{n} (\hat{\theta}_n - \theta^{\star} ) = - \bigg( \frac{1}{n} \sum_{i=1}^n \frac{\partial}{\partial \theta} \mathcal{A}_{\theta}f_{\hat{a}_n}^{\hat{b}_n} (X_i) \Big\vert_{\theta=\theta^{\star}} + A_n(\hat{\theta}_n,\hat{a}_n,\hat{b}_n) \bigg)^{-1} \frac{1}{\sqrt{n}} \sum_{i=1}^n \mathcal{A}_{\theta^{\star}}f_{\hat{a}_n}^{\hat{b}_n} (X_i).
\end{align*}
Consider the last term above, we apply the formula \eqref{discrete_taylor_lower} for $n=2$, $\tilde{k}=a^{\star}$, $k=\hat{a}_n$ and get
\begin{align} \label{equation_proof_asym_norm}
\begin{split}
    \frac{1}{\sqrt{n}} \sum_{i=1}^n \mathcal{A}_{\theta^{\star}}f_{\hat{a}_n}^{\hat{b}_n} (X_i) =& \frac{1}{\sqrt{n}} \sum_{i=1}^n \mathcal{A}_{\theta^{\star}}f_{a^{\star}}^{\hat{b}_n} (X_i) + (\hat{a}_n -a^{\star} ) \bigg( \frac{1}{\sqrt{n}} \sum_{i=1}^n \mathcal{A}_{\theta^{\star}}f_{a^{\star}+1}^{\hat{b}_n} (X_i) - \mathcal{A}_{\theta^{\star}}f_{a^{\star}}^{\hat{b}_n} (X_i) \bigg) \\
    & + \sum_{j=a^{\star}}^{\hat{a}_n-2} (\hat{a}_n-j-1) \bigg(\frac{1}{\sqrt{n}} \sum_{i=1}^n \mathcal{A}_{\theta^{\star}}f_{j+2}^{\hat{b}_n} (X_i) - 2\mathcal{A}_{\theta^{\star}}f_{j+1}^{\hat{b}_n} (X_i) +\mathcal{A}_{\theta^{\star}}f_{j}^{\hat{b}_n} (X_i)  \bigg).
\end{split}
\end{align}
We apply formula \eqref{discrete_taylor_upper} to the first term above and have with $n=2$, $\tilde{k}=b^{\star}$, $k=\hat{b}_n$ that
\begin{align*}
    \frac{1}{\sqrt{n}} \sum_{i=1}^n \mathcal{A}_{\theta^{\star}}f_{a^{\star}}^{\hat{b}_n} (X_i) =& \frac{1}{\sqrt{n}} \sum_{i=1}^n \mathcal{A}_{\theta^{\star}}f_{a^{\star}}^{b^{\star}} (X_i) - (b^{\star} - \hat{b}_n ) \bigg( \frac{1}{\sqrt{n}} \sum_{i=1}^n \mathcal{A}_{\theta^{\star}}f_{a^{\star}}^{b^{\star}} (X_i) - \mathcal{A}_{\theta^{\star}}f_{a^{\star}}^{b^{\star}-1} (X_i) \bigg) \\
    & + \sum_{j=\hat{b}_n+2}^{b^{\star}} (j-\hat{b}_n-1) \bigg(\frac{1}{\sqrt{n}} \sum_{i=1}^n \mathcal{A}_{\theta^{\star}}f_{a^{\star}}^{j} (X_i) - 2\mathcal{A}_{\theta^{\star}}f_{a^{\star}}^{j-1} (X_i) +\mathcal{A}_{\theta^{\star}}f_{a^{\star}}^{j-2} (X_i)  \bigg).
\end{align*}
Now, using Lemma \ref{lemma_conv_maxmin} together with Assumption \ref{assumptions_unknown_trunc} and the central limit theorem, it follows that
\begin{align*}
    \frac{1}{\sqrt{n}} \sum_{i=1}^n \mathcal{A}_{\theta^{\star}}f_{a^{\star}}^{\hat{b}_n} (X_i) = \frac{1}{\sqrt{n}} \sum_{i=1}^n \mathcal{A}_{\theta^{\star}}f_{a^{\star}}^{b^{\star}} (X_i) + o_{\mathbb{P}}(1),
\end{align*}
where $o_{\mathbb{P}}(1)$ denotes a term that converges to $0$ in probability. The same arguments can be used for all other terms in \eqref{equation_proof_asym_norm} to obtain
\begin{align*}
    \frac{1}{\sqrt{n}} \sum_{i=1}^n \mathcal{A}_{\theta^{\star}}f_{\hat{a}_n}^{\hat{b}_n} (X_i) = \frac{1}{\sqrt{n}} \sum_{i=1}^n \mathcal{A}_{\theta^{\star}}f_{a^{\star}}^{b^{\star}} (X_i) + o_{\mathbb{P}}(1).
\end{align*}
Likewise, we have
\begin{align*}
    \frac{1}{n} \sum_{i=1}^n \frac{\partial}{\partial \theta} \mathcal{A}_{\theta}f_{\hat{a}_n}^{\hat{b}_n} (X_i) \Big\vert_{\theta=\theta^{\star}} = \frac{1}{n} \sum_{i=1}^n \frac{\partial}{\partial \theta} \mathcal{A}_{\theta}f_{a^{\star}}^{b^{\star}} (X_i) \Big\vert_{\theta=\theta^{\star}} + o_{\mathbb{P}}(1).
\end{align*}
With Theorem \ref{theorem_consistency_unknown_trunc} and the assumptions made in the statement of the theorem it follows that
\begin{align*}
    \sqrt{n} (\hat{\theta}_n - \theta^{\star} ) = - \bigg( \frac{1}{n} \sum_{i=1}^n \frac{\partial}{\partial \theta} \mathcal{A}_{\theta}f_{a^{\star}}^{b^{\star}} (X_i) \Big\vert_{\theta=\theta^{\star}} +  o_{\mathbb{P}}(1) \bigg)^{-1} \frac{1}{\sqrt{n}} \sum_{i=1}^n \mathcal{A}_{\theta^{\star}}f_{a^{\star}}^{b^{\star}} (X_i) +  o_{\mathbb{P}}(1).
\end{align*}
and the claim now follows with Slutsky's lemma, the central limit theorem and the strong law of large numbers.
\end{proof}

Let us reconsider Example \ref{example_truncated_poisson} and \ref{example_truncated_multneg} with an unknown truncation domain. Simulation results for the Stein estimator and the MLE can be found in Tables \ref{truncpoissonunknowndom_sim} and \ref{truncnegmultunknowndom_sim}. For the $TNM(r,p,a,b)$-distribution we estimated $a=(a_1,\ldots,a_d)$ and $b=(b_1,\ldots,b_d)$ component-wise, i.e. for a sample $X_1,\ldots,X_n$, where $X_i=(X_i^{(1)},\ldots,X_i^{(d)})^{\top}$ we have
\begin{gather*}
    \hat{a}_n = \big(\min\{X_1^{(1)},\ldots, X_n^{(1)} \},\ldots, \min\{X_1^{(d)},\ldots, X_n^{(d)} \} \big), \\
    \hat{b}_n = \big(\max\{X_1^{(1)},\ldots, X_n^{(1)} \},\ldots, \max\{X_1^{(d)},\ldots, X_n^{(d)} \} \big).
\end{gather*}
For both distributions we observe that bias and MSE are slightly larger due to the additional uncertainty from the unknown truncation domain, although the differences are only minor. Interestingly, numerical optimization for the MLE appears to be more stable for some parameter constellation for $TNM(r,p,a,b)$ in the  case of an unknown truncation domain.

\begin{table} 
\centering
\begin{tabular}{cc|cc|cc|cc}
 $(\lambda^{\star},a^{\star},b^{\star})$ & & \multicolumn{2}{c|}{Bias} & \multicolumn{2}{c|}{MSE}  & \multicolumn{2}{c}{NE} \\ \hline
   & & $\hat{\theta}_n^{\mathrm{ML}}$ & $\hat{\theta}_n^{\mathrm{ST}}$ & $\hat{\theta}_n^{\mathrm{ML}}$ & $\hat{\theta}_n^{\mathrm{ST}}$ & $\hat{\theta}_n^{\mathrm{ML}}$ & $\hat{\theta}_n^{\mathrm{ST}}$ \\ \hline
\multirow{1}{*}{$(0.1,2,10)$} & $\lambda$  & $0.194$ & $0.029$ & $0.157$ & 6.28\text{e-3} & \multirow{1}{*}{0} & \multirow{1}{*}{19} \\  \hline 
\multirow{1}{*}{$(0.5,0,30)$} & $\lambda$  & $0.026$ & $0.029$ & $0.012$ & 0.013 & \multirow{1}{*}{0} & \multirow{1}{*}{0} \\  \hline 
\multirow{1}{*}{$(0.9,6,\infty)$} & $\lambda$  & $0.215$ & $3.58\text{e-3}$ & $0.09$ & 0.111 & \multirow{1}{*}{43} & \multirow{1}{*}{0} \\  \hline 
\multirow{1}{*}{$(1,0,80)$} & $\lambda$  & $0.038$ & $0.046$ & $0.024$ & 0.025 & \multirow{1}{*}{0} & \multirow{1}{*}{0} \\  \hline 
\multirow{1}{*}{$(1.5,6,85)$} & $\lambda$  & $0.118$ & $0.082$ & $0.195$ & 0.204 & \multirow{1}{*}{0} & \multirow{1}{*}{0} \\  \hline 
\multirow{1}{*}{$(2,2,40)$} & $\lambda$  & $0.075$ & $0.078$ & $0.083$ & 0.092 & \multirow{1}{*}{0} & \multirow{1}{*}{0} \\  \hline 
\multirow{1}{*}{$(2.5,6,90)$} & $\lambda$  & $0.146$ & $0.112$ & $0.264$ & 0.294 & \multirow{1}{*}{0} & \multirow{1}{*}{0} \\  \hline 
\multirow{1}{*}{$(3,0,50)$} & $\lambda$  & $0.057$ & $0.091$ & $0.068$ & 0.076 & \multirow{1}{*}{0} & \multirow{1}{*}{0} \\  \hline 
\multirow{1}{*}{$(3.5,1,10)$} & $\lambda$  & $0.074$ & $0.11$ & $0.087$ & 0.099 & \multirow{1}{*}{0} & \multirow{1}{*}{0} \\  \hline 
\multirow{1}{*}{$(4,0,20)$} & $\lambda$  & $0.049$ & $0.096$ & $0.091$ & 0.102 & \multirow{1}{*}{0} & \multirow{1}{*}{0} \\  \hline 
\end{tabular} 
\caption{\protect\label{truncpoissonunknowndom_sim} Simulation results for the $TP(\lambda,a,b)$ distribution with unknown truncation domain for $n=50$ and $10{,}000$ repetitions.}
\end{table}

\begin{table} 
\centering
\begin{tabular}{cc|cc|cc|cc}
 $(r,p^{\star},a^{\star},b^{\star})$ & & \multicolumn{2}{c|}{Bias} & \multicolumn{2}{c|}{MSE} & \multicolumn{2}{c}{NE} \\ \hline
& & $\hat{\theta}_n^{\mathrm{ML}}$ & $\hat{\theta}_n^{\mathrm{ST}}$ & $\hat{\theta}_n^{\mathrm{ML}}$ & $\hat{\theta}_n^{\mathrm{ST}}$ & $\hat{\theta}_n^{\mathrm{ML}}$ & $\hat{\theta}_n^{\mathrm{ST}}$ \\ \hline
\multirowcell{3}{$(5,(0.3,0.1,0.2) $, \\ $(0,0,0),(10,10,10))$} & $p_1$  & -- & 2.25\text{e-4} & -- & 7.8\text{e-4} & \multirow{3}{*}{100} & \multirow{3}{*}{0} \\ & $p_2$  & -- & 1.07\text{e-4} & -- & 1.52\text{e-4}\\ & $p_{3}$ & -- & 1.09\text{e-4} & -- & 4.04\text{e-4}\\ \hline 
\multirowcell{3}{$(5,(0.1,0.1,0.1) $, \\ $(1,1,2),(5,5,5))$} & $p_1$  & 2\text{e-3} & 3.77\text{e-3} & 1.42\text{e-4} & 2.13\text{e-3} & \multirow{3}{*}{6} & \multirow{3}{*}{1} \\ & $p_2$  & 2\text{e-3} & 3.17\text{e-3} & 1.45\text{e-4} & 2.15\text{e-3}\\ & $p_{3}$ & 1.6\text{e-3} & 4.76\text{e-3} & 2.18\text{e-4} & 2.4\text{e-3}\\ \hline 
\multirowcell{3}{$(7,(0.2,0.1,0.1) $, \\ $(0,0,0),(5,5,5))$} & $p_1$  & -2.2\text{e-3} & 2.28\text{e-3} & 2.13\text{e-4} & 1.31\text{e-3} & \multirow{3}{*}{89} & \multirow{3}{*}{0} \\ & $p_2$  & 1.98\text{e-3} & 4.57\text{e-4} & 8.7\text{e-5} & 2.82\text{e-4}\\ & $p_{3}$ & 1.93\text{e-3} & 4.07\text{e-4} & 9.04\text{e-5} & 2.92\text{e-4}\\ \hline 
\multirowcell{3}{$(2,(0.8,0.1,0.05) $, \\ $(0,0,2),(7,8,5))$} & $p_1$  & -8.14\text{e-3} & -0.153 & 2.46\text{e-3} & 0.058 & \multirow{3}{*}{63} & \multirow{3}{*}{40} \\ & $p_2$  & 3.98\text{e-3} & 3.08\text{e-3} & 9.91\text{e-5} & 1.82\text{e-3}\\ & $p_{3}$ & 4.09\text{e-3} & 1.7\text{e-3} & 1.45\text{e-4} & 5.31\text{e-4}\\ \hline 
\multirowcell{3}{$(5,(0.5,0.1,0.1) $, \\ $(2,1,1),(5,5,5))$} & $p_1$  & -1.19\text{e-4} & 8.35\text{e-3} & 2.36\text{e-3} & 0.032 & \multirow{3}{*}{28} & \multirow{3}{*}{9} \\ & $p_2$  & 1.29\text{e-3} & 1.29\text{e-3} & 1.31\text{e-4} & 8.38\text{e-4}\\ & $p_{3}$ & 1.44\text{e-3} & 8.69\text{e-4} & 1.32\text{e-4} & 8.18\text{e-4}\\ \hline 
\multirowcell{3}{$(3,(0.3,0.3,0.3) $, \\ $(0,0,0),(10,5,5))$} & $p_1$  & 2.59\text{e-4} & -3.55\text{e-4} & 1.51\text{e-7} & 6.66\text{e-4} & \multirow{3}{*}{100} & \multirow{3}{*}{0} \\ & $p_2$  & -0.012 & 4.32\text{e-3} & 2.87\text{e-4} & 2.4\text{e-3}\\ & $p_{3}$ & -0.017 & 4.06\text{e-3} & 5.68\text{e-4} & 2.37\text{e-3}\\ \hline 
\multirowcell{3}{$(5,(0.2,0.3,0.4) $, \\ $(1,1,1),(9,9,9))$} & $p_1$  & -- & 4.19\text{e-4} & -- & 2.74\text{e-4} & \multirow{3}{*}{100} & \multirow{3}{*}{0} \\ & $p_2$  & -- & 8.34\text{e-4} & -- & 9.06\text{e-4}\\ & $p_{3}$ & -- & 3.11\text{e-3} & -- & 2.23\text{e-3}\\ \hline 
\multirowcell{3}{$(1,(0.1,0.1,0.4) $, \\ $(0,0,0),(10,10,7))$} & $p_1$  & 4.75\text{e-3} & 0.015 & 4.64\text{e-4} & 7.71\text{e-3} & \multirow{3}{*}{15} & \multirow{3}{*}{13} \\ & $p_2$  & 5.02\text{e-3} & 0.016 & 4.64\text{e-4} & 7.53\text{e-3}\\ & $p_{3}$ & 7.2\text{e-3} & -0.032 & 1.2\text{e-3} & 0.031\\ \hline 
\multirowcell{3}{$(1,(0.4,0.5,0.05) $, \\ $(2,3,2),(10,10,10))$} & $p_1$  & 4.16\text{e-3} & 5.66\text{e-3} & 3.14\text{e-4} & 6.25\text{e-3} & \multirow{3}{*}{99} & \multirow{3}{*}{1} \\ & $p_2$  & -0.011 & 0.012 & 6.74\text{e-4} & 0.013\\ & $p_{3}$ & 5.46\text{e-3} & 7.26\text{e-4} & 1.61\text{e-4} & 1.82\text{e-4}\\ \hline 
\multirowcell{4}{$(1,(0.4,0.05,0.2,0.01) $, \\ $(1,1,0,0),(10,11,11,12))$} & $p_1$  & -0.017 & -0.056 & 2.89\text{e-4} & 0.014 & \multirow{4}{*}{100} & \multirow{4}{*}{93} \\ & $p_2$  & 0.032 & 0.087 & 1.03\text{e-3} & 0.022\\ & $p_{3}$ & -0.02 & -3.51\text{e-4} & 3.97\text{e-4} & 0.016\\ & $p_{4}$ & 0.228 & 0.018 & 0.052 & 1.32\text{e-3}\\ \hline 
\end{tabular} 
\caption{\protect\label{truncnegmultunknowndom_sim} Simulation results for the $TNM(r,p,a,b)$ distribution with unknown truncation domain for $n=100$ and $10{,}000$ repetitions.}
\end{table}

\section*{Acknowledgements}
AF is funded by EPSRC Grant EP/T018445/1. 

\bibliography{library}
\bibliographystyle{abbrv}

\end{document}